\documentclass[12pt,reqno]{amsart}
\usepackage[T1]{fontenc}
\usepackage[english]{babel}
\usepackage[utf8]{inputenc}
\usepackage[square,numbers]{natbib}
\usepackage{amssymb}
\usepackage{amsmath}
\usepackage{xcolor}
\usepackage{amsfonts}
\usepackage{amsthm}
\usepackage{csquotes}
\usepackage[a4paper,margin=2.5cm]{geometry}
\usepackage{graphicx} 
\usepackage{comment}
\usepackage{tikz}
\usepackage{subcaption}
\usepackage{enumitem}
\usepackage{float}
\usepackage{textgreek}

\usepackage[colorlinks]{hyperref}
\hypersetup{
 colorlinks = true,
 linkcolor = blue,
}
\setcitestyle{numbers}
\theoremstyle{plain}

\newcommand{\var}{\operatorname{var}}

\newcommand{\norm}[1]{\|#1\|}
\newcommand{\R}{\mathbb{R}}
\newcommand{\C}{\mathbb{C}}
\newcommand{\N}{\mathbb{N}}
\newcommand{\tr}{\operatorname{tr}}

\newcommand{\A}{\mathcal{A}}

\newcommand{\cumuf}{\kappa^{\operatorname{free}}}

\newcommand{\NC}{\operatorname{NC}}
\newcommand{\homo}{\operatorname{hom}}
\newcommand{\Homo}{\operatorname{Hom}}

\title[Graphon-Theoretic Approach to CLT for $\epsilon$-Independence]{Graphon-Theoretic Approach to Central Limit Theorems for $\epsilon$-Independence}
\author[G. Cébron and P.\ O.\ Santos \and P.\ Youssef]{%
        Guillaume Cébron \and
        Patrick Oliveira Santos \and 
        Pierre Youssef
        }

\address{Guillaume Cébron. Institut de Mathématiques de Toulouse; UMR5219; Université de Toulouse; CNRS; UPS, F-31062 Toulouse, France}
\email{guillaume.cebron@math.univ-toulouse.fr}
\address{Patrick Oliveira Santos. Division of Science, NYU Abu Dhabi, Abu Dhabi, UAE.}
\email{po2150@nyu.edu}
\address{Pierre Youssef. Division of Science, NYU Abu Dhabi, Abu Dhabi, UAE \& Courant Institute of Mathematical Sciences, New York University, New York, USA.}
\email{yp27@nyu.edu}

\date{\today}

\begin{document}
\newtheorem{theorem}{Theorem}[section]

\newtheorem{corollary}[theorem]{Corollary}
\newtheorem{lemma}[theorem]{Lemma}
\newtheorem{conjecture}[theorem]{Conjecture}
\newtheorem{proposition}[theorem]{Proposition}

\theoremstyle{definition}
\newtheorem{example}[theorem]{Example}
\newtheorem{definition}[theorem]{Definition}
\newtheorem{remark}[theorem]{Remark}

\newtheorem*{model}{Model}

\maketitle
\begin{abstract}
We establish a central limit theorem for the sum of $\epsilon$-independent random variables, extending both the classical and free probability setting. Central to our approach is the use of graphon limits to characterize the limiting distribution, which depends on the asymptotic structure of the underlying graphs governing $\epsilon$-independence. This framework yields a range of exotic limit laws, interpolating between free and classical cases and allowing for mixtures such as free and classical convolutions of the semi-circle and Gaussian distributions. We provide a complete characterization of the limiting law, captured as the distribution of an operator on the full Fock space. We extend our main result to the multivariate setting as well as the non-identically distributed case. 
The proof provides insights into the combinatorial structure of $\epsilon$-independence, shedding light on a natural connection between the graph of crossings of the partitions involved and the graphon limit of the graphs of $\epsilon$-independence.
\end{abstract}
\section{Introduction}

In the landscape of noncommutative probability, classical and free independence have long stood as the only universal notions that satisfy commutativity and independence from constants. Both forms of independence manifest in distinct product constructions — tensor and reduced free product — and provide the foundation for much of probability theory in operator algebras. The quest for more nuanced forms of independence, particularly mixtures of classical and free structures, presents an enticing and challenging extension. 
The seminal work of M\l{}otkowski \cite{mlotkowski2004lambdafree} first introduced the concept of $\Lambda$-freeness, providing a mixture of classical and free independence. While this opened new doors, it left aspects of the combinatorial structure and analytical properties of such mixed systems ripe for exploration. Later, Speicher and Wysoczański \cite{speicherjanusz2016mixture} (see also \cite{ebrahimifard2017}) offered a combinatorial framework for $\Lambda$-freeness, renaming it $\epsilon$-independence, elaborating the corresponding moment-cumulant formulas. Note also the operator algebraic study of this particular independence by Caspers and
Fima \cite{caspers2017graph}, and the various random matrix models which asymptotically yield this independence \cite{morampudi2019many,charlesworth2021matrix,magee2308strongly,charlesworth2024random}. In view of the close relationship with graph products of groups developed by Green \cite{green1990graph}, $\epsilon$-independence is also called graph independence, or right-angled probability, in the literature.

Central limit theorems (CLTs) have played a crucial role in both classical and free probability theory. The classical CLT reveals the universality of the Gaussian distribution as it appears as the limit of normalized sums of classically independent random variables. In the 1980s, Voiculescu \cite{voiculescu1985symmetries} established an analogous result in the free probability context, known as the Free CLT, where the normalized sum of freely independent self-adjoint random variables converges to a semi-circular distribution, which is central to free probability and its applications in random matrix theory.

Between these two regimes, where variables are neither fully independent nor free, a natural question is to understand how limit theorems manifest in this mixed setting. Moreover, it is interesting to capture the universal limiting distribution and understand the structural dependencies between the variables that result in a particular limiting measure. While both classical and free CLTs have been well-studied, CLTs within $\epsilon$-independence remain largely uncharted. To the best of our knowledge, existing results have typically focused on particular cases, such as the distribution of $q$-Gaussians \cite{boejko1997}. These limited cases hint at deeper structures but lack a comprehensive framework for understanding the general behavior of sums of $\epsilon$-independent variables. 
The goal of this paper is to address this and establish a general central limit theorem in the setting of $\epsilon$-independence. 
To achieve this, we encode the evolution of $\epsilon$-independence through graphon convergence \cite{lovasz2012largenetworks} and show that the normalized sum of $\epsilon$-independent converges to a universal limit that interpolates between the classical Gaussian and semi-circular distributions. The latter interpolation is explicitly encoded using the graphon limit of the sequence of graphs of $\epsilon$-independence. 

Interestingly, graphon limits have been used in the context of random matrix theory to encode the evolution of variance profiles and a corresponding limit theorem for the empirical spectral distribution of non-homogeneous Wigner matrices established accordingly~\cite{zhu2019}. 
It is also worth noting that other notions of independence, such as mixtures of free, monotone, and boolean independence, have been explored, with a corresponding central limit theorem established (see \cite{ARIZMENDI2025110712,jekel2024general} and references therein). 

Recall that a graphon consists of a measurable bounded symmetric function $w:[0,1]^2 \to [0,1]$. Given a graph $g$ on $n$ vertices, one can naturally associate a graphon $w_g$ defined by $w_g(x,y)=A_g(i,j)$  for $i,j \in [n]$ and
\begin{align*}
    x \in \left[\frac{i-1}{n},\frac{i}{n}\right),\quad y \in \left[\frac{j-1}{n},\frac{j}{n}\right),
\end{align*}
where $A_g$ is the adjacency matrix of $g$. Given a graphon $w$ and a graph $f=([n],E)$ in $n$ nodes, we define
\begin{align*}
    \rho(f,w):=\int_{[0,1]^n} \prod_{(i,j) \in E} w(x_i,x_j) \, \text{d}x.
\end{align*}
When the graphon $w$ is generated by a graph $g$, i.e., $w=w_g$, the above definition reduces to 
\begin{align*}
    \rho(f,g):=\rho(f,w_g)=\frac{\homo(f,g)}{|V(g)|^{|V(f)|}},
\end{align*}
where $\homo(f,g)$ is the number of homomorphism from $f$ to $g$, namely, $\rho(f,g)$ is the probability that a uniformly chosen random function from $V(f)$ into $V(g)$  is an homomorphism. 
A sequence of simple graphs $g_n$ in $n$ nodes converges to a graphon $w$ if, for every simple finite graph $f$, we have
\begin{align*}
    \rho(f,g_n) \to \rho(f,w).
\end{align*}

Let $(\A,\tau)$ be a unital noncommutative probability space equipped with a faithful tracial state $\tau$ \cite{nica2006lectures} i.e. $\tau:\, \A\to \C$ is a linear functional satisfying $\tau(1)=1$ and $\tau(a^*a)\geq 0$ for every $a\in \A$. Within this framework, classical independence can be generalized as follows: a family $\{\A_i\}_{i\in \N}$ of unital subalgebras of $\A$ is called independent if the algebras commute (meaning $ab=ba$ whenever $a\in \A_i$ and $b\in A_j$ with $i\neq j$), and if for any $m\in \N$ and any sequence of elements $a_1\in \A_{i_1}, \ldots, a_m\in \A_{i_m}$, we have 
$\tau(a_1\ldots a_m)=0$ whenever $\tau(a_1)=\cdots=\tau(a_m)=0$ and $i_k\neq i_l$ for $k\neq l$. 

In contrast, Voiculescu’s concept of freeness replaces commutativity with a different type of independence. A family $\{\A_i\}_{i\in \N}$ of unital subalgebras of $\A$ is called free if, for any sequence of elements $a_1\in \A_{i_1}, \ldots, a_m\in \A_{i_m}$ with $\tau(a_1)=\cdots=\tau(a_m)=0$ and $i_1\neq i_2\neq \cdots\neq i_m$, we have $\tau(a_1\ldots a_m)=0$. This form of independence, unlike the classical case, is not associated with the tensor product but rather with the unital free product of algebras.

As we previously mentioned, the concept of $\epsilon$-independence provides a mixture of freeness and classical independence. Formally, given a loopless graph $g_n=([n],E)$, we say that subalgebras $\A_1,\ldots,\A_n \subset \A$ are $g_n$-independent if
\begin{enumerate}
    \item For every $(i,j) \in E$, the subalgebras $\A_i,\A_j$ are classical independent (in particular, they commute);
    \item For every $k \ge 1$, $i \in [n]^k$ such that for all $j_1<j_2$ with $i_{j_1}=i_{j_2}$, there exists $j_1<j_3<j_2$ with $(i_{j_1},i_{j_3}) \notin E$, then for any centered random variables $a_j \in \A_{i_j}$, $j \in [k]$, we have
    \begin{align*}
        \tau(a_1\cdots a_k)=0.
    \end{align*}
\end{enumerate}

Note that when the graph $g_n$ is complete, we recover classical independence, while when the graph $g_n$ is edgeless, we recover the notion of freeness. 

We say that
a sequence of (self-adjoint) variables $(a_n)_{n\in \N}$ in $\A$  converge in distribution to $a\in \A$ if, for every $p\in \N$, $\tau(a_n^p) \to \tau(a^p)$ as $n\to \infty$. With these notations, the classical (resp. free) CLT states that given $a, a_1,\ldots,a_n \in \A$ classically independent (resp. free) and identically distributed centered random variables, their normalized sums
\begin{align}\label{eq: normalized sums}
    S_n:=\frac{1}{\sqrt{n}}\sum_{k \in [n]}a_k
\end{align}
converge in distribution to a Gaussian (resp. semi-circle) random variable with mean zero and variance equal to the variance of $a$. 
The combinatorial proof of those results involves the understanding of pair partitions. While all pairings play a role in the classical CLT, only noncrossing pairings contribute to the free CLT. For a pair partition $\pi\in P_2(2p)$, we define its intersection graph $f_\pi$ whose vertex set consists of the blocks of $\pi$, and there exists an edge between two blocks if they cross. Noncrossing pair partitions are those whose intersection graph consists of isolated vertices (see Section~\ref{sec: fock construction} for more details).

With all definitions in hand, we can now state the main result of this paper, which encompasses both classical and free CLT as particular instances of a general central limit theorem for $\epsilon$-independent variables.

\begin{theorem}\label{theorem: 1}
    Let $(g_n)_{n \ge 1}$ be a sequence of finite simple graphs where $V(g_n)=[n]$. Assume that $g_n \to w$ in the graphon sense. Let $a_1^{(n)},\ldots,a_n^{(n)} \in \A$ be $g_n$-independent identically distributed centered (self-adjoint) random variables with variance one, for every $n \ge 1$. Then 
    \begin{align*}
        S_n=\frac{1}{\sqrt{n}}\sum_{k \in [n]}a_k^{(n)}
    \end{align*}
    converges in distribution to a law $\mu_w$ that depends only on $w$. Moreover, its odd moments vanish, whereas, for any $p\in \N$, we have
    \begin{align*}
        \int x^{2p}\, \text{d}\mu_w=\sum_{\pi \in P_2(2p)}\rho(f_\pi,w).
    \end{align*}
\end{theorem}

In the setting of the classical CLT, i.e., when the $g_n$'s are complete graphs, the corresponding graphon limit is the constant graphon $w(x,y)=1$ for all $x\neq y$. One can readily see that Theorem~\ref{theorem: 1} applied in this setting states that the even moments of the limiting law are given by the number of pair partitions coinciding with that of the Gaussian distribution. 
Now, in the setting of the free CLT, i.e., when the $g_n$'s are edgeless graphs, the corresponding graphon limit is the constant graphon $w(x,y)=0$ for all $x\neq y$. It is easy to see that Theorem~\ref{theorem: 1} applied in this setting recovers the semi-circle law as the limiting distribution. In addition to recovering those classical results, Theorem~\ref{theorem: 1} asserts the stability of the classical CLT (resp. free CLT) when the $g_n$'s are sufficiently dense (resp. sparse). We refer to Section~\ref{sec: examples} and Examples~\ref{ex: semi-circle} and \ref{ex: gaussian} for precise statements.

The class of universal laws appearing at the limit in Theorem~\ref{theorem: 1} is quite vast, and closed under classical and free convolution (see Proposition~\ref{proposition: mu_w is closed under cla/free}). The measures $\mu_w$ are universal objects, always arising as the limiting measures in CLTs for $\epsilon$-independent variables. This universality stems from the fact that every graphon $w$ can be realized as the limit of some sequence of (random) graphs (see \cite[Corollary 2.6]{lovasz2006limits}), making $\mu_w$ the limiting measure associated with $\epsilon$-independent random variables defined by those (random) graphs. The breadth of this class reflects the diversity of graphon limits, encompassing an extensive and exotic variety of distributions. We refer to Section~\ref{sec: examples} for examples. 
One particular example which represents a relatively well-studied class of distributions interpolating between the Gaussian and semi-circle distributions is the class of $q$-Gaussians for $q\in [0,1]$ \cite{boejko1997}, whose $2p$-th moments are given by 
\begin{equation}\label{eq: q-gaussian def}
    \sum_{\pi \in P_2(2p)}q^{\vert E(f_\pi)\vert}.
\end{equation}
Theorem~\ref{theorem: 1} recovers these distributions and the corresponding central limit theorem \cite[Corollary 3]{mlotkowski2004lambdafree}, as they coincide with $\mu_w$ for the constant graphon $w(x,y)=q$ for all $x\neq y$. 
It is worth noting that $q$-Gaussians represent in some sense an extremal case of Theorem~\ref{theorem: 1} as every limiting law $\mu_w$ appearing in Theorem~\ref{theorem: 1} is stochastically dominated (in terms of moments) by some $q$-Gaussian distribution. This is formalized in the following proposition. 

\begin{proposition}\label{proposition: bound on the norm of the limit}
    Let $w$ be a graphon and $\mu_w$ be the measure such that its odd moments vanish, and for any integer $p\ge 1$, we have
    \begin{align*}
        \int x^{2p}\, \text{d}\mu_w=\sum_{\pi \in P_2(2p)}\rho(f_\pi,w).
    \end{align*}
Denote
    \begin{align*}
        q:=\sup_{(x,y)\in [0,1]^2}w(x,y),
    \end{align*}
and let $\mu_q$ be the corresponding $q$-Gaussian law \eqref{eq: q-gaussian def}. Then, for all integers $p\ge 1$, we have that  
\begin{align*}
        \int x^{p}\, \text{d}\mu_w \le \int x^{p}\, \text{d}\mu_q.
    \end{align*}
In particular, if $q<1$, then the measure $\mu_w$ is bounded and its support is contained in $[-2/\sqrt{1-q},2/\sqrt{1-q}]$. 
\end{proposition}

While the previous proposition reveals probabilistic properties of the measures $\mu_w$'s, we also explore the algebraic aspect. Specifically, we construct a representation of $\mu_w$ as the distribution of a particular operator on the full Fock space. These constructions will hopefully allow for further exploration of operator-valued random variables associated with the $\epsilon$-independence framework, connecting them to other algebraic constructions in free probability. We refer to Section~\ref{sec: fock construction} where we carry this construction and establish some properties of the measures $\mu_w$'s, in particular, that they are absolutely continuous with respect to the Lebesgue measure.

Furthermore, we extend Theorem~\ref{theorem: 1} by exploring several generalizations. In Section~\ref{sec: non i.i.d case}, we generalize Theorem~\ref{theorem: 1} beyond the identically distributed case and establish a weighted CLT for $\epsilon$-independence. This generalization requires dealing with the weighted graphon theory \cite{lovasz2008convergence1,lovasz2012convergence2}. In Section~\ref{sec: multivariate case}, we establish the multivariate version of Theorem~\ref{theorem: 1}, which requires using a notion of ``joint'' graphon convergence, provided by the concept of decorated graphons \cite{lovasz2022multigraph}. Finally, in Section~\ref{sec: negative graphon}, we explore formulations of Theorem~\ref{theorem: 1} covering the case of $q$-Gaussians with $-1\leq q\leq 0$. Given the technical sophistication required to address these generalizations rigorously, we defer the precise statements and proofs to the corresponding sections. The proof of Theorem~\ref{theorem: 1} is carried out in Section~\ref{sec: proof of main th}. 

\subsection*{Acknowledgments} 
This work was initiated during the workshop ``Random matrices and Free probability'' held in June 2024 at Institut de Math\'ematiques de Toulouse. We would like to thank Mireille Capitaine, in particular, for the stimulating program.

\section{Fock representation of $\mu_w$ and first properties}\label{sec: fock construction}

\subsection{Preliminaries}We start this section by recalling some standard notions of free probability and graphon theory.
Given $p\in \N$, a partition $\pi=\{V_1, \ldots ,V_m\}$ of $[p]$ is a collection of disjoint sets $V_1,\ldots,V_m$ called blocks such that 
\begin{align*}
    V_1 \cup \cdots \cup V_m=[p].
\end{align*}
We denote by $\vert \pi\vert$ the number of blocks of $\pi$. We denote by $P(p)$ the set of partitions of $[p]$ and $P_2(p)$ the set of all pair partitions of $[p]$. Here, we define 
\begin{align*}
    \sum_{\pi\in P_2(p)}t(\pi):=0,
\end{align*}
for odd integers $p$ and function $t:P_2(p)\to \R$.
We say that two blocks $V_1$ and $V_2$ of a partition $\pi \in P(p)$ cross if there exist $i<k<j<l$ such that $\{i,j\}\subset V_1$ and $\{k,l\} \subset V_2$. 
We say that a partition $\pi \in P(p)$ is a \textit{noncrossing} partition if none of its blocks cross each other. 
We denote by $NC(p)$ the set of all noncrossing partitions of $[p]$ and $\NC_2(p)$ the set of all noncrossing pair partitions of $[p]$. 
Finally, recall that for a partition $\pi \in P(p)$, we denote by $f_\pi$ its \textit{intersection graph}. It is the graph over the blocks of $\pi$ such that two blocks are connected if they cross, under some arbitrary labeling. We say that a partition $\pi$ is connected whenever $f_\pi$ is; see Figures \ref{fig: partitions} and \ref{fig: partitions gc}.
\begin{figure}[H]
     \centering
        \begin{subfigure}[b]{0.3\textwidth}
         \centering
         \begin{tikzpicture}[scale=0.50]
            \draw[-] (1,0) edge (1,1);
            \draw[-] (1,1) edge (3,1);
            \draw[-] (3,1) edge (3,0);
        
            \draw[-] (2,0) edge (2,2);
            \draw[-] (2,2) edge (5,2);
            \draw[-] (5,2) edge (5,0);
        
            \draw[-] (4,0) edge (4,1);
            \draw[-] (4,1) edge (6,1);
            \draw[-] (6,1) edge (6,0);
        \end{tikzpicture}
        \caption{Connected partition}
    \end{subfigure}
     \hfill
     \begin{subfigure}[b]{0.3\textwidth}
         \centering
         \begin{tikzpicture}[scale=0.50]
            \draw[-] (1,0) edge (1,1);
            \draw[-] (1,1) edge (3,1);
            \draw[-] (3,1) edge (3,0);
        
            \draw[-] (2,0) edge (2,2);
            \draw[-] (2,2) edge (4,2);
            \draw[-] (4,2) edge (4,0);
        
            \draw[-] (5,0) edge (5,1);
            \draw[-] (5,1) edge (6,1);
            \draw[-] (6,1) edge (6,0);
        \end{tikzpicture}
         \caption{General partition}
     \end{subfigure}
     \hfill
     \begin{subfigure}[b]{0.3\textwidth}
         \centering
         \begin{tikzpicture}[scale=0.50]
            \draw[-] (1,0) edge (1,2);
            \draw[-] (1,2) edge (4,2);
            \draw[-] (4,2) edge (4,0);
        
            \draw[-] (2,0) edge (2,1);
            \draw[-] (2,1) edge (3,1);
            \draw[-] (3,1) edge (3,0);
        
            \draw[-] (6,0) edge (6,1);
            \draw[-] (6,1) edge (5,1);
            \draw[-] (5,1) edge (5,0);
        \end{tikzpicture}
         \caption{Noncrossing partition}
     \end{subfigure}
        \caption{Examples of partitions.}
        \label{fig: partitions}
\end{figure}
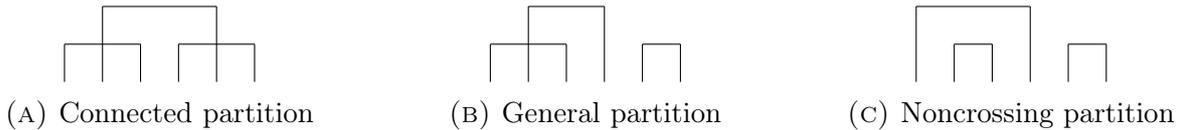
\begin{figure}[H]
\centering
     \begin{subfigure}[b]{0.3\textwidth}
         \centering
         \begin{tikzpicture}[scale=0.50]
            \node at (0,0) (A) {};
            \fill[black] (A) circle(1.5pt);
            \node at (1,1) (B) {};
            \fill[black] (B) circle(1.5pt);
            \node at (2,0) (C) {};
            \fill[black] (C) circle(1.5pt);

            \draw[-] (0,0) edge (1,1);
            \draw[-] (1,1) edge (2,0);
        \end{tikzpicture}
        \caption{Connected partition}
    \end{subfigure}
    \hfill
     \begin{subfigure}[b]{0.3\textwidth}
         \centering
         \begin{tikzpicture}[scale=0.50]
            \node at (0,0) (A) {};
            \fill[black] (A) circle(1.5pt);
            \node at (1,1) (B) {};
            \fill[black] (B) circle(1.5pt);
            \node at (2,0) (C) {};
            \fill[black] (C) circle(1.5pt);

            \draw[-] (0,0) edge (1,1);
        \end{tikzpicture}
         \caption{General partition}
     \end{subfigure}
     \hfill
     \begin{subfigure}[b]{0.3\textwidth}
         \centering
         \begin{tikzpicture}[scale=0.50]
            \node at (0,0) (A) {};
            \fill[black] (A) circle(1.5pt);
            \node at (1,1) (B) {};
            \fill[black] (B) circle(1.5pt);
            \node at (2,0) (C) {};
            \fill[black] (C) circle(1.5pt);
        \end{tikzpicture}
         \caption{Noncrossing partition}
     \end{subfigure}
        \caption{Examples of intersection graphs.}
        \label{fig: partitions gc}
\end{figure}
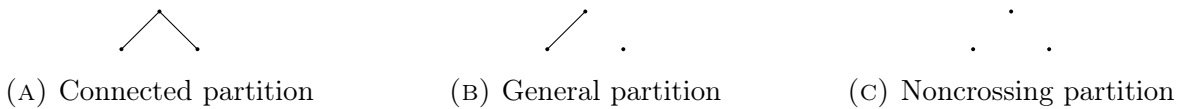

Given $\pi \in P(p)$ and $u,v\in [p]$, we denote $u\sim_\pi v$ if $u,v$ are in the same block of $\pi$. 
For $n,p \ge 1$, and $i\in [n]^p$, we denote $\ker(i) \in P(p)$ the partition such that $i_u=i_v$ if and only if $u$ and $v$ belong to the same block of $\ker(i)$. Now given $g_n=([n],E)$ a simple graph on $n$ vertices, we denote $\NC(g_n,i)$ the set of $(g_n,i)$-noncrossing partitions, namely, those partitions $\pi \in P(k)$ such that $\pi \le \ker(i)$ (that is, $\pi$ only connects nodes that are connected by $\ker (i)$), and  
\begin{equation}\label{eq: def-non-crossing eps partitions}
\text{if }\, u_1<u_2<v_1<v_2 \text{ such that ${u_1} \nsim_\pi {u_2}$, ${u_1} \sim_\pi {v_1}$, ${u_2}\sim_\pi {v_2}$,\, then $(i_{u_1}i_{u_2}) \in E$.} 
\end{equation}
Denote $\cumuf_n(a_1,\ldots,a_n)$ the free cumulants of $(a_1,\ldots,a_n)$. 

\subsection{Fock representation of $\mu_w$}
Given any symmetric measurable function $w:[0,1]^2\to [-1,1]$,  we denote by $\mu_{w}$ the probability measure whose moments are given by 
\begin{align}\label{eq: definition of mu_w}
   \sum_{\pi \in P_2(2p)}\int_{[0,1]^{V(f_\pi)}} \prod_{(u,v) \in E(f_\pi)} w(x_u,x_v) \, \text{d}x, 
\end{align}
where $f_\pi=(V(f_\pi),E(f_\pi))$ is the intersection graph of $\pi$, and the odd moments vanish. 
In this section, we define a representation of $\mu_w$ as the distribution of some operator acting on the full Fock space. We will carry the construction for graphons taking values in $[-1,1]$ rather than only $[0,1]$ as in Theorem~\ref{theorem: 1}. We refer to Section~\ref{sec: negative graphon} where we investigate the corresponding limit theorem.

The construction is inspired by that of the $q$-Brownian motion given by
creation and annihilation operators in \cite{bozejko1991example}. In order to get the positive definiteness of the "twisted" scalar product, we will use the very general results of \cite{bozejko1994completely}, which deal with the positivity of maps on the symmetric group defined from an operator satisfying the braid (or Yang-Baxter) relations.

Given $w$, we define the operator $T_w:L^2([0,1]^2)\to L^2([0,1]^2)$ by
$$T_wh(x,y)=w(x,y)h(y,x),$$
for any $h\in L^2([0,1]^2)$. We will use $\otimes$ to denote the tensor product of Hilbert spaces in such a way that $(L^2([0,1]))^{\otimes n}$ can always be identified with $L^2([0,1]^n)$ for $n\geq 1$. 
Consequently, given the Hilbert space $H:=L^2([0,1])$, $T_w$ can be seen as a self-adjoint contraction $T_w\in B(H\otimes H)$ with $\|T_w\|_\infty=\sup_{x,y\in [0,1]} |w|$, and satisfying the Yang-Baxter equation
$$(id_H\otimes T_w)\cdot(T_w\otimes id_H)\cdot(id_H\otimes T_w)=(T_w\otimes id_H)\cdot(id_H\otimes T_w)\cdot(T_w\otimes id_H),$$
where $id_H$ denotes the identity operator on $H$. Indeed, 
for any $h\in L^2([0,1]^3)$,
\begin{align*}
   & [(id_H\otimes T_w)\cdot(T_w\otimes id_H)\cdot(id_H\otimes T_w)h](x,y,z)\\
   &=h(z,y,x)w(x,y)w(y,z)w(z,x)\\
   &=[(T_w\otimes id_H)\cdot(id_H\otimes T_w)\cdot(T_w\otimes id_H)h](x,y,z).
\end{align*}Then, for any $i\geq 1$, we define
$$T_i:=id_H\otimes \cdots \otimes id_H\otimes T_w$$
acting on $H^{\otimes i+1}$ (note that $T_1=T_w$). By amplification, $T_i$ is also acting on $H^{\otimes n}$ for all~$n\geq i+1$. Consider for fixed $n\in \mathbb{N}$ the permutation group $S_n$ and denote by
$\sigma_i \in S_n$ ($i = 1, . . . , n - 1$) the transposition between $i$ and $i + 1$. Since the $T_i$'s satisfy the braid relations, we can define the function
$\varphi_n:S_n\to B(H^{\otimes n})$
by quasi-multiplicative extension of $\varphi_n(\sigma_i)=T_i$ (see \cite{bozejko1994completely}). More precisely, for a reduced word $\sigma = \sigma_{i_1}\cdots \sigma_{i_k}\in S_n$ with $1\leq i_1,\ldots,i_k\leq n-1$, we set $\varphi_n(\sigma)=T_{i_1}\cdots T_{i_k}$.

On the full Fock space $$\mathcal{F}=\bigoplus_{n=0}^\infty H^{\otimes n}\ \ (\text{where}\ H^{\otimes 0}=\mathbb{C}\Omega),$$
we define the symmetric bilinear form $\langle ,\rangle_{T_w}$ given by
$$\langle \xi,\eta\rangle_{T_w}:=\delta_{nm}\langle \xi,P^{(n)}\eta\rangle $$ for $\xi \in H^{\otimes n}$, $\eta\in H^{\otimes m}$ where
$$P^{(n)}:=\sum_{\sigma \in S_n} \varphi(\sigma) $$
is the canonic operator corresponding to $\varphi$. According to \cite[Theorem 2.2]{bozejko1994completely} the operators $P^{(n)}$ are positive, thus $\langle ,\rangle_{T_w}$ is positive
definite. If $\|T_w\| < 1$ then, by \cite[Theorem 2.3]{bozejko1994completely}, we have that all the $P^{(n)}$'s are strictly positive
and we can take $\mathcal{F}_{T_w}$ as the completion of $\mathcal{F}$ with respect to $\langle ,\rangle_{T_w}$. When
$\|T_w\| = 1$, the kernel might not be empty and the completion is made after taking the corresponding quotient.

Now we define, for each $h\in L^2([0,1])$, the canonic creation operator $c(h)$ by 
\begin{align*}
    c(h)\cdot g=h\otimes g\in L^2([0,1])^{\otimes (n+1)},
\end{align*}
for any $g\in L^2([0,1])^{\otimes n}$. Denote by $c(h)^*$ the adjoint of $c(h)$ with respect to $\langle ,\rangle_{T_w}$. 
For any $h\in L^2([0,1])$, we define $\mu_{w,h}$ the probability measure whose moments are given by 
\begin{equation}
    \int x^p\,\text{d}\mu_{w,h}:=\sum_{\pi \in P_2(p)}\int_{[0,1]^{V(f_\pi)}} \prod_{u\in V(f_\pi)}|h(x_u)|^2\prod_{(u,v) \in E(f_\pi)} w(x_u,x_v) \, \text{d}x,\label{eq:moments}
\end{equation}
Note that $\mu_w=\mu_{w,1}$. 

\begin{theorem}
Let $h\in L^2([0,1])$ and $A(h)= c(h)+c(h)^*$. 
The moments of $A(h)$, with respect to the vector state $\tau:=\langle \cdot \Omega ,\Omega \rangle_{T_w}$, coincide with those of $\mu_{w,h}$. 
In particular, $\mu_w$ is the law of $A(1)$ with respect to $\tau$. 
\end{theorem}
\begin{proof}
We will first show that for any $h\in L^2([0,1])$, the adjoint $c(h)^*$ of $c(h)$ with respect to $\tau$ is given, for any~$g\in L^2([0,1])^{\otimes n}$, by
 \begin{equation}\label{eq:cfs}
     [c(h)^*\cdot g](x_1,\ldots,x_{n-1})=\sum_{k=1}^n\int_0^1 \prod_{\ell=1}^{k-1}w(x_\ell,y)\cdot \overline{h(y)}\cdot g(x_1,\ldots,x_{k-1},y,x_k,\ldots,x_{n-1})dy.
 \end{equation}
To this aim, for any $h\in L^2([0,1])$, let us denote by $a(f)$ the left annihilation operator given, for any~$g\in L^2([0,1])^{\otimes n}$, by
    $$[a(h)\cdot g](x_1,\ldots,x_{n-1})=\int_0^1 \overline{h(y)}\cdot g(y,x_1,\ldots,x_{n-1})dy.$$
        From \cite[Theorem 3.1]{bozejko1994completely}, we know that
        $$c(h)^*=a(h)\cdot [1+T_1+T_1T_2+\cdots +T_1T_2\cdots T_{n-1}]$$
        on $H^{\otimes n}$.
  By induction on $j$ decreasing from $k$ to $1$, we have, for any $1\leq j\leq k \leq n-1$ and $g\in L^2([0,1])^{\otimes n}$,
    $$[T_{j}T_{j+1}\cdots T_kg](x_1,\ldots,x_n)=\prod_{\ell=j+1}^{k}w(x_j,x_{\ell})\cdot g(x_1,\ldots,x_{j-1},x_{j+1},\ldots,x_{k+1},x_j,x_{k+2},\ldots,x_n),$$
        which implies that
            $$[T_1T_2\cdots T_kg](x_1,\ldots,x_n)=\prod_{\ell=2}^{k}w(x_1,x_{\ell})\cdot g(x_2,,\ldots,x_{k+1},x_1,x_{k+2},\ldots,x_n),$$
            and
$$
    [a(h)\cdot (T_1\cdots T_k)g](x_1,\ldots,x_{n-1})
   = \int_0^1 \prod_{\ell=2}^{k}w(y,x_{\ell-1})\cdot \overline{h(y)}\cdot g(x_1,\ldots,x_{k},y,x_{k+1},\ldots,x_{n-1})dy.
 $$
 Finally, we get
     \begin{align}&[c(h)^*\cdot g](x_1,\ldots,x_n)
    \nonumber\\&=[a(f)\cdot (1+T_1+T_1T_2+\cdots +T_1T_2\cdots T_{n-1})g](x_1,\ldots,x_n)\nonumber\\
    &=\sum_{k=1}^n\int_0^1 \prod_{\ell=1}^{k-1}w(x_\ell,y)\cdot \overline{h(y)}\cdot g(x_1,\ldots,x_{k-1},y,x_k,\ldots,x_{n-1})dy,
     \end{align}
which establishes \eqref{eq:cfs}.      
We are now ready to compute
$ (c(h)+c(h)^*)^{p}\Omega$. We consider the set $P_{1,2}(p)$ of incomplete pairings of $[p]$, i.e., partitions of $[p]$ into singletons and pairs (incomplete pairings are sometimes called Feynman diagrams, see \cite{janson1997gaussian}). For any $\pi\in P_{1,2}(p)$, we have as before the intersection graph $f_{\pi}$ whose vertices are the blocks of~$\pi$. However, we need to specify the rules for counting crossings for singletons. Vertices that are not coupled should be singletons and
are drawn with straight lines to the top. The intersection graph $f_{\pi}$ is the graph over the blocks of $\pi$ such that two
blocks are connected if they cross: note that in addition to the crossings between two pairs of $\pi$, there exists an edge between a singleton $\{i\}\in \pi$ and a pair $\{a<b\}\in \pi$ whenever $a<i<b$, according to our convention. We denote by $p(\pi)$ the set of pairings in $\pi$ and by $s(\pi)$ the singletons in $\pi$. We have
$$  (c(h)+c(h)^*)^{p}\Omega=\sum_{\pi \in P_{1,2}(p)}h_{\pi},$$
where for any $\pi \in P_{1,2}(p)$ with $\ell$ singletons and $m$ pairings, $h_{\pi}$ is the element of $L^2([0,1])^{\ell}$ given by
$$h_{\pi}(x_i:i\in s(\pi))=\int_{[0,1]^{m}}\prod_{(u,v)\in E(f_{\pi})}w(x_u,x_v)\prod_{i\in s(\pi)}h(x_i)\prod_{j\in p(\pi)} |h(x_j)|^2dx_j.$$
Such a formula is proven by induction on $p$: we have $(c(h)+c(h)^*)\Omega=h=\sum_{\pi \in P_{1,2}(1)}h_{\pi}$, and, for $p\geq 1$,
\begin{align*}
(c(h)+c(h)^*)^{p+1}\Omega&=\sum_{\pi \in P_{1,2}(p)}c(h)h_\pi+\sum_{\pi \in P_{1,2}(p)}c(h)^*h_\pi\\
&=\sum_{\pi \in P_{1,2}(p)}\sum_{\substack{\tilde{\pi} \in P_{1,2}(p+1)\\ \{1\}\in \tilde{\pi}\text{ and } \tilde{\pi}|_{[p+1]\setminus \{1\}}=\pi}}h_{\tilde{\pi}}+\sum_{\pi \in P_{1,2}(p)}\sum_{\substack{\tilde{\pi} \in P_{1,2}(p+1)\\ \{1\}\notin \tilde{\pi}\text{ and } \tilde{\pi}|_{[p+1]\setminus \{1\}}=\pi}}h_{\tilde{\pi}}\\
&=\sum_{\tilde{\pi} \in P_{1,2}(p+1)}h_{\tilde{\pi}},
\end{align*}
where the explicit formula~\eqref{eq:cfs} of $c(h)^*$ is used to write
$$c(h)^*h_\pi=\sum_{\substack{\tilde{\pi} \in P_{1,2}(p+1)\\ 1\text{ is coupled with }\{2,\cdots,p+1\}\text{ in }\tilde{\pi}\\ \tilde{\pi}|_{[p+1]\setminus \{1\}}=\pi}}h_{\tilde{\pi}}=\sum_{\substack{\tilde{\pi} \in P_{1,2}(p+1)\\ \{1\}\notin \tilde{\pi}\text{ and } \tilde{\pi}|_{[p+1]\setminus \{1\}}=\pi}}h_{\tilde{\pi}}.$$
Finally, in the sum
$$  A(h)^{p}\Omega=\sum_{\pi \in P_{1,2}(p)}h_{\pi},$$
the only elements which are in $H^{\otimes 0}=\mathbb{C}\Omega$ are the $h_\pi$'s where $\pi$ is a complete pairing, i.e. $\pi\in  P_2(p)$. We get $$\langle A(h)^p\Omega,\Omega\rangle=\sum_{\pi \in \mathcal{P}_{2}(p)}h_{\pi},$$
which yields $\eqref{eq:moments}$.
\end{proof}

    \subsection{First properties}
We start by providing the free cumulants of the measures $\mu_{w,h}$'s.    
\begin{proposition}\label{proposition: free cumulants of mu_w}
The $p$-th free cumulant of $\mu_{w}$ is vanishing if $p$ is odd and is equal to
$$\cumuf_p(\mu_{w})=\sum_{\pi \in P_2^{con}(p)}\rho(f_\pi,w)$$
if $p$ is even, where $P_2^{con}(k)$ is the set of connected pair partitions of $P_2(p)$ (the partition for which the intersection graph has only one connected component). More generally, the $p$-th free cumulant of $\mu_{w,h}$ is vanishing if $p$ is odd and is equal to
$$\cumuf_p(\mu_{w,h})=\sum_{\pi \in P_2^{con}(p)}\int_{[0,1]^{V(f_\pi)}}\prod_{u\in V(f_\pi)}|h(x_u)|^2\prod_{(u,v) \in E(f_\pi)} w(x_u,x_v) \, \text{d}x$$
if $p$ is even.
\end{proposition}
\begin{proof}
    It is a consequence (following \cite[Proof of Theorem 4.1]{lehner2002free}) of the fact that we have
multiplicativity of the term $$\int_{[0,1]^{V(f_\pi)}}\prod_{u\in V(f_\pi)}|h(x_u)|^2\prod_{(u,v) \in E(f_\pi)} w(x_u,x_v) \, \text{d}x$$ over the connected components of any pair partition.
\end{proof}

Let $t\mu$ be the dilation of the measure $\mu$ by $t>0$, namely,
\begin{align*}
    t\mu(A)=\mu(t^{-1}A),
\end{align*}
for any Borel set $A$.
\begin{proposition}
Up to dilations, the set of probability measures $\mu_{w}$ is the same as the set of probability measures $\mu_{w,h}$.
\end{proposition}
\begin{proof}
Let us prove that $\mu_{w,h}$ is equal to a certain $\mu_{w'}$ (up to a dilation). Note that the dilation acts as 
$t\mu_{w,h}=\mu_{w,th}.$ By dilating $\mu_{w,h}$ by $1/\|h\|_{L^2}$ if necessary, we can assume that $|h|^2$ is a probability density on $[0,1]$. We denote by $G:]0,1[\to [0,1]$ the pseudo-inverse of the cumulative distribution function of $|h|^2$, in such a way that, for any uniform random variable $U$ on $[0,1]$, the density of $G(U)$ is given by $|h|^2$. We consider the graphon $w':[0,1]^2\to [-1,1]$ given by
$$w'(x,y)=w(G(x),G(y)).$$
Therefore, we can write
\begin{align*}
     \int x^{p}\, \text{d}\mu_{w'}&=\sum_{\pi \in P_2(p)}\int_{[0,1]^{V(f_\pi)}} \prod_{(u,v) \in E(f_\pi)} w(G(x_u),G(x_v)) \, \text{d}x\\
     &=\sum_{\pi \in P_2(p)}\int_{[0,1]^{V(f_\pi)}} \prod_{(u,v) \in E(f_\pi)} w(x_u,x_v) \prod_{u\in V(f_\pi)}|h(x_u)|^2\, \text{d}x\\
     &= \int x^{p}\, \text{d}\mu_{w,h}.\qedhere
\end{align*}
\end{proof}
The family of measures $\mu_w$ is considerably large. In particular, the following proposition shows that it is closed under classical and free convolution.
\begin{proposition}\label{proposition: mu_w is closed under cla/free}
Up to dilations, the set of probability measures $\mu_{w}$ is closed by taking classical convolution and by taking free convolution. 
\end{proposition}
\begin{proof}
    We first prove that for any (positive) graphons $w,w':[0,1]^2\to  [0,1]$ and $\lambda\in (0,1)$, there exists another graphon $w^*$ such that
    \begin{align*}
        \sqrt{\lambda}\mu_w+\sqrt{1-\lambda}\mu_{w'}=\mu_{w^*}.
    \end{align*}
    Let $n_\lambda=\lceil \lambda n\rceil$. By \cite[Theorem 4.8]{borgs2006counting}, we can find two sequences of graphs $g_n,g'_n$ over $n_\lambda$ and $n-n_\lambda$ vertices, respectively, such that $g_n\to w$ and $g'_n \to w'$. We label their vertices as
    \begin{align*}
        &V(g_n)=[n_\lambda];\\
        &V(g'_n)=\{n_\lambda+1,\ldots,n\}.
    \end{align*}
    Consider then the graph $g^*_n$ over $[n]$ defined as 
    \begin{align*}
        E(g^*_n)=\{(i,j)\in E(g_n)\}\sqcup \{(i,j)\in E(g'_n)\}&\sqcup\{(i,j)\in V(g_n)\times V(g'_n)\}\\
        &\sqcup\{(i,j)\in V(g'_n)\times V(g_n)\}.
    \end{align*}
    Informally, $g^*_n$ has two disjoint clusters $g_n$ and $g'_n$, and all vertices in $g_n$ are connected to all vertices in $g'_n$. We can check that the graphon $w_{g^*_n}$ converges to 
    \begin{align}\label{eq: graphon convolution classical}
        w^*=\begin{cases}
            w\left(\frac{x}{\lambda},\frac{y}{\lambda}\right);&\text{ if } (x,y)\in [0,\lambda)^2;\\
            w'\left(\frac{x-\lambda}{1-\lambda},\frac{y-\lambda}{1-\lambda}\right);& \text{ if }(x,y)\in [\lambda,1]^2\\
            1;& \text{ otherwise.}
        \end{cases}
    \end{align}
    Let $a_n$ be $g^*_n$-independent and identically distributed random variables with common law $a$, where $a$ is centered with variance one. Then, Theorem \ref{theorem: 1} implies that the law of
    \begin{align*}
        S_n=\frac{1}{\sqrt{n}}\sum_{k\in [n]}a_k 
    \end{align*}
    converges to $\mu_{w^*}$. On the other hand, we decompose
    \begin{align*}
        S_n=\sqrt{\frac{n_\lambda}{n}}\frac{1}{\sqrt{n_\lambda}}\sum_{k\in V(g_n)}a_k+\sqrt{\frac{n-n_\lambda}{n}}\frac{1}{\sqrt{n-n_\lambda}}\sum_{k\in V(g'_n)}a_k.
    \end{align*}
    The variables $a_k$ for $k\in V(g_n)$ are $g_n$-independent, and the variables $a_k$ for $k\in V(g'_n)$ are $g'_n$-independent. Finally, between them, they are classically independent. As $n_\lambda/n\to \lambda$, we deduce by Theorem \ref{theorem: 1} that the law of $S_n$ converges to
    \begin{align*}
        \sqrt{\lambda}\,\mu_w+\sqrt{1-\lambda}\,\mu_{w'}.
    \end{align*}
    Hence
    \begin{align*}
        \mu_{w^*}=\sqrt{\lambda}\,\mu_w+\sqrt{1-\lambda}\,\mu_{w'},
    \end{align*}
    by Carleman's condition.
    For the free convolution of $\mu_w$ and $\mu_{w'}$, consider again the graphs $g_n\to w$ and $g'_n\to w'$ defined above. We define the graph $\overline{g}^*_n$ over $[n]$ such that
    \begin{align*}
        E(\overline{g}^*_n)=\{(i,j)\in E(g_n)\}\sqcup \{(i,j)\in E(g'_n)\}.
    \end{align*}
    It follows that $w_{\overline{g}^*_n}$ converges towards
        \begin{align}\label{eq: graphon convolution free}
        \overline{w}^*=\begin{cases}
            w\left(\frac{x}{\lambda},\frac{y}{\lambda}\right);&\text{ if } (x,y)\in [0,\lambda)^2;\\
            w'\left(\frac{x-\lambda}{1-\lambda},\frac{y-\lambda}{1-\lambda}\right);& \text{ if }(x,y)\in [\lambda,1]^2;\\
            0; &\text{ otherwise.}
        \end{cases}
    \end{align}
    By an argument similar to the classical case, we can see that
    \begin{align*}
        \mu_{\overline{w}^*}=\sqrt{\lambda}\,\mu_w \boxplus \sqrt{1-\lambda}\,\mu_{w'}.
    \end{align*}
    The case of negative graphons follows similarly by making use of Theorem \ref{theorem: clt for negative graphons} (see Section~\ref{sec: negative graphon}) in place of Theorem~\ref{theorem: 1}.
\end{proof}
\begin{remark}\label{remark: construction convolution}
    Proposition \ref{proposition: mu_w is closed under cla/free} can also be proved by a direct computation using the definitions of $w^*$ and $\overline{w}^*$ in \eqref{eq: graphon convolution classical} and \eqref{eq: graphon convolution free}. For the classical convolution, this computation is reminiscent of the proof of \cite[Proposition 3.1]{lancien2024centrallimittheoremtensor}.
\end{remark}
The explicit construction of $\mu_{w}$ on the Fock space allows to get a concrete formula for a dual system (in the sense of Voiculescu~\cite{voiculescu1998analogues}), following the work of Miyagawa and Speicher~\cite{miyagawa2023dual} concerning $q$-Gaussians. It yields the following absolute continuity of $\mu_{w}$ (we refer to~\cite{miyagawa2023dual} and references therein for the definition of the free Fisher information and the definition of a dual system).
\begin{proposition}
Suppose that
\begin{align*}
    \sup_{(x,y)\in [0,1]^2}|w(x,y)|<1.
\end{align*}
Then, the free Fisher information of $\mu_{w}$ is finite. In particular, $\mu_{w}$ is absolutely continuous with respect to the Lebesgue measure, and its density is in $L^3(\mathbb{R})$. 
\end{proposition}
\begin{proof}
The proof follows the lines of that in \cite{miyagawa2023dual} with a few adaptations, which we outline below while omitting the lengthy details. 
The main ingredient is the explicit construction of a dual system for~$A(1)$ on the Fock space. As in~\cite{miyagawa2023dual}, the finiteness of the free Fisher information is a consequence of the existence of this dual system, thanks to \cite[Theorem~1]{shlyakhtenko2006remarks}. Moreover, by~\cite[Proposition 8.18]{mingo2017free}, the finiteness of the free Fisher information implies the absolute continuity with respect to the Lebesgue measure, and the fact that the density is in $L^3(\mathbb{R})$.

In order to define the dual operator of $A(1)$, we use below the notation of~\cite{miyagawa2023dual}. In particular, $B(n + 1)$ denotes a collection of incomplete pairings of $\{n>n-1>\cdots >1>0\}$
introduced in~\cite[Section 4]{miyagawa2023dual}, where the rules of counting crossings differ from the usual definition of crossings in $P_{1,2}(n+1)$. The new intersection graph $\tilde{f}_{\pi}$ of a partition $\pi \in B(n + 1)$ is the graph over the blocks of $\pi$ such that two
blocks are connected if they cross according to these new rules. For $\pi\in B(n+1)$, we denote by $p(\pi)$ the set of pairings in $\pi$, by $s(\pi)$ the singletons in $\pi$ and by $\pi(0)$ the integer which is coupled to $0$ in $\pi$.

Following~\cite[Proposition 4.5]{miyagawa2023dual}, we define an unbounded operator $D$ with the domain $\mathcal{F}=\bigoplus_{n=0}^\infty H^{\otimes n}$ by linear extension of $D\Omega=0$ and $Dh =\sum_{\pi \in B(n+1)}h_{\pi}$ whenever $h\in L^2([0,1]^n)$, where for any $\pi \in B(n+1)$ with $\ell$ singletons and $m$ pairings, $h_{\pi}$ is the element of $L^2([0,1])^{\ell}$ given by
$$h_{\pi}(x_i:i\in s(\pi))=(-1)^{\pi(0)-1}\int_{[0,1]^{m}}h(x_{\bar{n}},\ldots,x_{\bar{1}})\prod_{(u,v)\in E(\tilde{f}_{\pi})}w(x_u,x_v)\prod_{j\in p(\pi)}dx_j$$
(note that all the variables $x_i$ and $x_j$ are indexed by the blocks of $\tilde{f}_\pi$, and the blocks $\bar{n},\ldots,\bar{1}$ denote respectively the blocks of $n,\ldots,1$ in $\tilde{f}_\pi$). Using the explicit formula~\eqref{eq:cfs} of~$c(h)^*$ and following the proof of~\cite[Proposition 4.5]{miyagawa2023dual}, we get that the operator $D$ satisfies
$$D\cdot A(1)-A(1)\cdot D=P$$where $P$ is the orthogonal projection onto $\mathbb{C}\Omega$. In order to conclude, it remains to prove that $\Omega$ is in the domain of $D^*$, and again, the proof follows the one of~\cite[Theorem 4.6]{miyagawa2023dual}: it consists in showing that the linear functional $\langle D\ \cdot\ ,\Omega\rangle_{T_{w}}$ is bounded by computing $\langle Dh,\Omega\rangle_{T_{w}}$ for $h$ being any finite linear combination of pure vectors, and bounding it by a constant times $\|h\|_{T_{w}}$. One crucial ingredient in the proof is the fact that the free right annihilation operators are bounded on $\mathcal{F}=\bigoplus_{n=0}^\infty H^{\otimes n}$. This fact, which is the content of \cite[Lemma~2.2]{miyagawa2023dual}, remains valid in our context since~\cite[Theorem~6]{bozejko1998completely}, used to prove \cite[Lemma~2.2]{miyagawa2023dual}, is stated in a level of generality which applies to our setting.
\end{proof}
\section{Proof of Theorem~\ref{theorem: 1}}\label{sec: proof of main th}

 The following result provides the moment-cumulant formula for $\epsilon$-independence~\cite[Theorem~5.2]{speicherjanusz2016mixture}.
\begin{lemma}\label{lemma: speicher}
    Let $n,p \ge 1$, $g_n=([n],E)$ be a simple graph in $n$ nodes and $i \in [n]^p$. Let $\A_1,\ldots,\A_n$ be $g_n$-independent subalgebras, and $a_j \in \A_{i_j}$, for $j \in [p]$. Then
    \begin{align*}
        \tau(a_1\cdots a_p)=\sum_{\pi \in \NC(g_n,i)} \prod_{V=\{v_1,\ldots,v_m\} \in \pi} \cumuf_m(a_{v_1},\ldots,a_{v_m}).
    \end{align*}
\end{lemma}
We are ready to prove Theorem \ref{theorem: 1}. To simplify the notation, we will omit the superscript $a_k:=a_k^{(n)}$.
\begin{proof}[Proof of Theorem \ref{theorem: 1}]
    We begin by writing
    \begin{align*}
        \tau(S_n^p)=\frac{1}{n^{p/2}}\sum_{i \in [n]^p}\tau(a_{i_1}\cdots a_{i_p}).
    \end{align*}
    We decompose the summation $i \in [n]^p$ according to the partition $\ker(i) \in P(p)$,
    \begin{align*}
        \tau(S_n^p)=\frac{1}{n^{p/2}}\sum_{\pi \in P(p)}\sum_{\substack{i \in [n]^p\\ \ker(i)=\pi}}\tau(a_{i_1}\cdots a_{i_p}).
    \end{align*}
    By Lemma \ref{lemma: speicher}, we have
    \begin{align*}
         \tau(a_{i_1}\cdots a_{i_p})=\sum_{\sigma \in NC(g_n,i)}\cumuf_\sigma(a_{i_1},\ldots,a_{i_p}).
    \end{align*}
    Let $\pi=\ker(i) \in P(p)$, and $\sigma \in \NC(g_n,i)$. By definition, we have $\sigma \le \pi$. If there exists $V=\{v\}\in \pi$, then we also have that $V \in \sigma$. In this case, we have
    \begin{align*}
        \cumuf_\sigma(a_{i_1},\ldots,a_{i_p})=\cumuf_1(a_{i_v})\cumuf_{\sigma\setminus\{V\}}\left((a_{i_j})_{j \ne v}\right)=0,
    \end{align*}
    since $\cumuf_1(a_{i_v})=\tau(a)=0$. Therefore, $\tau(a_{i_1}\cdots a_{i_p})=0$.
    We can then consider only the contribution of partitions $\pi \in P(p)$ such that $|V| \ge 2$ for any block $V \in \pi$. Let us denote them by $P_{\ge 2}(p)$. Then
    \begin{align*}
        \tau(S_n^p)=\frac{1}{n^{p/2}}\sum_{\pi \in P_{\ge 2}(p)}\sum_{\substack{i \in [n]^p\\ \ker(i)=\pi}}\tau(a_{i_1}\cdots a_{i_p}).
    \end{align*}
    Note that 
    \begin{align*}
        \#\{i \in [n]^p: \ker(i)=\pi\} \le n^{|\pi|}.
    \end{align*}
    By the Cauchy-Schwarz inequality, we have
    \begin{align*}
        \tau(a_{i_1}\cdots a_{i_p}) \le \tau(a^{2p})^{\frac{1}{2}}=:C_p,
    \end{align*}
    We deduce that
    \begin{align*}
        \frac{1}{n^{p/2}}\sum_{\substack{i \in [n]^p\\ \ker(i)=\pi}}\tau(a_{i_1}\cdots a_{i_p}) \le C_pn^{|\pi|-p/2},
    \end{align*}
    Therefore, if there exists $V \in \pi$ such that $|V| \ge 3$, we have $|\pi|<p/2$ and its contribution is negligible. Thus, we deduce that
    \begin{align*}
        \lim_{n \to \infty} \tau(S_n^p)=\sum_{\pi \in P_2(p)}\lim_{n \to \infty} \frac{1}{n^{p/2}}\sum_{\substack{i \in [n]^p\\ \ker(i)=\pi}}\tau(a_{i_1}\cdots a_{i_p}),
    \end{align*}
    as long as the limit on the right-hand side exists. We write again by Lemma \ref{lemma: speicher}
    \begin{align*}
        \tau(a_{i_1}\cdots a_{i_p})=\sum_{\sigma \in NC(g_n,i)}\cumuf_\sigma(a_{i_1},\ldots,a_{i_p}).
    \end{align*}
    
    Let $\sigma \in \NC(g_n,i)$ and $\ker(i)=\pi \in P_2(p)$. By definition, $\sigma \le \pi$ (that is, if $k,l$ are in the same block of $\sigma$, then they are also in the same block of $\pi$). Since the variables $a_{i_l}$ are centered, whenever there exists a single block $V=\{v\}\in \sigma$, we have
    \begin{align*}
        \cumuf_\sigma(a_{i_1},\ldots,a_{i_p})=\cumuf_1(a_{i_v})\cumuf_{\sigma\setminus\{V\}}\left((a_{i_j})_{j \ne v}\right)=0,
    \end{align*}
    since $\cumuf_1(a_{i_v})=\tau(a)=0$. Let us denote
    \begin{align*}
        \NC_{\ge 2}(g_n,i)=\{\sigma \in \NC(g_n,i): |V| \ge 2, \ \forall\, V \in \sigma\}.
    \end{align*}
    Then, we deduce that
    \begin{align*}
        \tau(a_{i_1}\cdots a_{i_p})=\sum_{\sigma \in NC_{\ge 2}(g_n,i)}\cumuf_\sigma(a_{i_1},\ldots,a_{i_k}).
    \end{align*}
    However, as $\pi \in P_2(p)$ and any $\sigma \in \NC_{\ge 2}(g_n,i)$ satisfies $\sigma \le \pi$, we must have $\sigma=\pi$. That is, we either have $\NC_{\ge 2}(g_n,i)=\varnothing$ or $\pi \in \NC(g_n,i)$ and $\NC_{\ge 2}(g_n,i)=\{\pi\}$.
    In the latter, if $\pi \in \NC(g_n, i)$, we get that
    \begin{align*}
        \cumuf_\pi(a_{i_1},\ldots,a_{i_p})=\prod_{V \in \pi} \cumuf_2(a)=1.
    \end{align*}
    Therefore, for $\ker(i)=\pi \in P_2(p)$, we have
    \begin{align*}
        \tau(a_{i_1}\cdots a_{i_p})=\mathbf{1}_{\pi \in NC(g_n,i)},
    \end{align*}
    and
    \begin{align*}
        \lim_{n \to \infty} \tau(S_n^p)=\sum_{\pi \in P_2(p)}\lim_{n \to \infty} \frac{1}{n^{p/2}}\#\{i \in [n]^p: \ker(i)=\pi \in \NC(g_n,i)\}.
    \end{align*}
    Recall that we define the intersection graph $f_\pi$ of $\pi$ by the graph over the blocks of $\pi$, and there is an edge between two blocks if they cross. We immediately see that
    \begin{align*}
        n^{p/2}=|V(g_n)|^{|V(f_\pi)|}.
    \end{align*}
    Moreover, if $V_1=\{v_1,u_1\} \in \pi$ crosses $V_2=\{v_2,u_2\}\in \pi_2$, the condition $\pi \in \NC(g_n,i)$ implies that $(i_{v_1}i_{v_2}) \in E$ (see \eqref{eq: def-non-crossing eps partitions}). In other words, if $(V_1V_2)$ is an edge in $f_\pi$, then $(i_{v_1}i_{v_2})$ is also an edge in $g_n$. Therefore, the map $V=\{v_1,u_1\} \mapsto i_{v_1}=i_{u_1}$ is an (injective) homomorphism from the graph $f_\pi$ to $g_n$. We then have that
    \begin{multline*}
        \{i \in [n]^p: \ker(i)=\pi \in \NC(g_n,i)\}\\=\{\phi: V(f_\pi)\to V(g_n): \phi \text{ is an injective homomorphism}\}.
    \end{multline*}
    Let us denote
    \begin{align*}
        \rho_{\operatorname{inj}}(f,g)=\frac{1}{|V(g)|^{|V(f)|}}\#\{\phi: V(f)\to V(g): \phi \text{ is an injective homomorphism}\}.
    \end{align*}
    Then
    \begin{align*}
        \lim_{n \to \infty} \tau(S_n^p)=\sum_{\pi \in P_2(p)}\lim_{n \to \infty} \rho_{\operatorname{inj}}(f_\pi,g_n).
    \end{align*}   
    As $n\to \infty$, the homomorphism densities $\rho(f_\pi,g_n)$ and $\rho_{\operatorname{inj}}(f_\pi,g_n)$ are indistinguishable (see Inequality 5.21 in \cite{lovasz2012largenetworks} and Lemma \ref{lemma: diff between homo and homoinj} below) and we have
    \begin{align*}
        |\rho(f_\pi,g_n)-\rho_{\operatorname{inj}}(f_\pi,g_n)|\le \frac{C_{f_\pi}}{n}\to 0,
    \end{align*}
    where $C_{f_\pi}$ is a constant that depends only on $f_\pi$.
    We deduce that
    \begin{align*}
        \lim_{n \to \infty} \tau(S_n^p)=\sum_{\pi \in P_2(p)}\lim_{n \to \infty} \rho(f_\pi,g_n).
    \end{align*}
    Since $g_n \to w$, we deduce that
    \begin{align*}
        \lim_{n \to \infty} \tau(S_n^p)=\sum_{\pi \in P_2(p)} \rho(f_\pi,w),
    \end{align*}
    and the result follows.
\end{proof}
\begin{remark}
    The usual definition of $\rho_{\operatorname{inj}}(f,g)$ requires the normalization
    \begin{multline*}
        \hat{\rho}_{\operatorname{inj}}(f,g)\\=\frac{1}{n(n-1)\cdots (n-k+1)}\#\{\phi: V(f)\to V(g): \phi \text{ is an injective homomorphism}\},
    \end{multline*}
    where $n=|V(g)|$ and $k=|V(f)|$. However, both definitions are asymptotically equivalent as 
    \begin{align*}
        \lim_{n\to \infty}\frac{n^k}{n(n-1)\cdots (n-k+1)}=1,
    \end{align*}
    for any $k$ fixed.
\end{remark}

\begin{proof}[Proof of Proposition \ref{proposition: bound on the norm of the limit}]
    By definition of $\mu_w$ and $q$, we get that for any even $p$ 
    \begin{align*}
        \int x^p\,\text{d}\mu_w=\sum_{\pi \in P_2(p)}\int_{[0,1]^{V(f_\pi)}}\prod_{(u,v)\in E(f_\pi)}w(x_u,x_v)\,\text{d}x\le \sum_{\pi \in P_2(p)}\int_{[0,1]^{V(f_\pi)}}\prod_{(u,v)\in E(f_\pi)}q\,\text{d}x.
    \end{align*}
    Hence
    \begin{align*}
        \int x^p\,\text{d}\mu_w\le \sum_{\pi \in P_2(p)}q^{|E(f_\pi)|}=\int x^p\,\text{d}\mu_q,
    \end{align*}
    where $\mu_q$ is the $q$-Gaussian distribution. It follows that
    \begin{align*}
        \mu_w^\infty\le \mu_q^\infty,
    \end{align*}
    where for a measure $\mu$, we denote
    \begin{align*}
        \mu^\infty:=\lim_{p\to \infty}\left(\int x^{2p}\,\text{d}\mu\right)^{1/2p}.
    \end{align*}
When $q<1$, the measure $\mu_q$ is bounded \cite{boejko1997,bozejko1991example,bozejko1994completely}, and its support is contained in 
\begin{align*}
    I_q:=[-2/\sqrt{1-q},2/\sqrt{1-q}].
\end{align*}
    Since the support of a measure $\mu$ is contained in the interval $[-\mu^\infty,\mu^\infty]$, we deduce that $\mu_w$ is also bounded and its support is included in $I_q$. 
\end{proof}

\section{The non-i.i.d case}\label{sec: non i.i.d case}
The goal of this section is to extend Theorem \ref{theorem: 1} to the case of non-identically distributed random variables. 

We define a (node)-weighted graph $(g_n,\alpha_n)$ as a pair consisting of a graph $g_n=([n],E(g_n))$ and a node-weight $\alpha_n: V(g_n)\to [0,\infty)$.
Let
\begin{align*}
    \vert\alpha_n\vert :=\sum_{i\in [n]}\alpha_n(i).
\end{align*}
For a simple graph $f$, we define the homomorphism density 
\begin{align*}
    \rho(f,(g_n,\alpha_n))=\frac{1}{\vert \alpha_n\vert^{|V(f)|}}\sum_{\phi \in \Homo(f,g_n)}\prod_{v\in V(f)}\alpha_n(\phi(v)),
\end{align*}
where $\Homo(f,g_n)$ is the set of all homomorphism functions from $V(f)$ to $V(g_n)$. The graph parameter $\rho(f,g)$ can be recovered when we set $\alpha\equiv 1$. 
Moreover, we can define the graphon $w_{g,\alpha}$ by setting
\begin{align*}
    &w_{g,\alpha}(x,y)=A_{g}(i,j),
\end{align*}
whenever
\begin{align*}
    &\alpha(1)+\cdots+\alpha(i-1)\le \vert \alpha\vert  x<\alpha(1)+\cdots+\alpha(i)\\
    &\alpha(1)+\cdots+\alpha(j-1)\le \vert \alpha\vert y<\alpha(1)+\cdots+\alpha(j).
\end{align*}
It is not difficult to check that $\rho(f,w_{g,\alpha})=\rho(f,(g,\alpha))$.

Let $\lambda$ be the Lebesgue measure in $[0,1]$ and let $\mathcal{M}$ be the set of all couplings of $(\lambda,\lambda)$, namely, all measures $\mu$ in $[0,1]^2$ whose marginals are $\lambda$. Given two graphons $w_1$ and $w_2$, we define the (general) cut distance $\delta_\Box (w_1,w_2)$  \cite{lovasz2008convergence1,lovasz2012convergence2}
\begin{align*}
    \delta_\Box (w_1,w_2)=\inf_{\mu \in \mathcal{M}}\sup_{S\times T\subseteq [0,1]^2}\left|\int_{S\times T}\left(w_1(x,y)-w_2(u,v)\right)\mathbf{1}_{(x,u)\in T}\mathbf{1}_{(y,v)\in S}\text{d}\mu(x,u)\text{d}\mu(y,v)\right|.
\end{align*}
It was proved in \cite{lovasz2008convergence1} that $\rho(f,(g_n,\alpha_n))$ converges if and only if $w_{g_n,\alpha_n}$ is a Cauchy sequence in the metric $\delta_\Box$. In this case, there exists a limit $w$ of the graphons $w_{g_n,\alpha_n}$ \cite[Corollary~3.9]{lovasz2008convergence1} and we have
\begin{align*}
    \rho(f,(g_n,\alpha)) \to \rho(f,w),
\end{align*}
for any simple graph $f$. Before stating the main result of this section, we need the following lemma. Let $\rho_{\operatorname{inj}}(f,(g,\alpha))$ be the (injective) homomorphism density
\begin{equation}\label{eq: def inj-hom-weight}
    \rho_{\operatorname{inj}}(f,(g,\alpha))=\frac{1}{\vert \alpha\vert^{|V(f)|}}\sum_{\phi \in \Homo_{\operatorname{inj}}(f,g)}\prod_{v\in V(f)}\alpha(\phi(v)),
\end{equation}
where $\Homo_{\operatorname{inj}}(f,g)$ is the set of injective homomorphism maps from $V(f)$ to $V(g)$. For a node-weight map $\alpha:V(g)\to [0,\infty)$, we define
\begin{align*}
    \alpha_{\max}=\max_{v\in V(g)}\alpha(v).
\end{align*}
\begin{lemma}\label{lemma: diff between homo and homoinj}
    For any simple graph $f$ and node-weighted graph $(g,\alpha)$ with $|V(g)|\ge |V(f)|$, we have
    \begin{align*}
        |\rho(f,(g,\alpha))-\rho_{\operatorname{inj}}(f,(g,\alpha))|\le \frac{C_f\alpha_{\max}}{\vert\alpha\vert},
    \end{align*}
    where $C_f$ is a constant that depends only on $f$.
\end{lemma}
\begin{proof}
    Let $k=|V(f)|$ and $n=|V(g)|$. Denote by $\Homo_{\operatorname{ninj}}(f,g)$ the set of non-injective homomorphism maps from $V(f)$ to $V(g)$. Hence
    \begin{align*}
        \gamma:&=\sum_{\phi \in \Homo(f,g)}\prod_{v\in V(f)}\alpha_n(\phi(v))-\sum_{\phi \in \Homo_{\operatorname{inj}}(f,g)}\prod_{v\in V(f)}\alpha_n(\phi(v))\\
        &=\sum_{\phi \in \Homo_{\operatorname{ninj}}(f,g)}\prod_{v\in V(f)}\alpha_n(\phi(v)).
    \end{align*}
    Note that any $\phi\in \Homo_{\operatorname{ninj}}(f,g)$ fixes at least two points, namely, there exist distinct $v_1,v_2\in V(f)$ such that $\phi(v_1)=\phi(v_2). $
    Therefore, the set $\{\phi(v):v\in V(f)\}$ has at most $k-1$ elements. Hence,
    \begin{align*}
        \gamma\le C_f\alpha_{\max}\sum_{v\in [n]^{k-1}}\prod_{l\in [k-1]}\alpha_n(v_l),
    \end{align*}
    where $C_{f}=\binom{|V(f)|}{2}$ is the number of ways of picking $v_1,v_2$. We note that
    \begin{align*}
        \sum_{v\in [n]^{k-1}}\prod_{l\in [k-1]}\alpha_n(v_l)=\vert\alpha\vert^{k-1}.
    \end{align*}
    Therefore, we conclude that
    \begin{equation*}
        |\rho(f,(g,\alpha))-\rho_{\operatorname{inj}}(f,(g,\alpha))|= \frac{\gamma}{\vert\alpha\vert^{k}}\le \frac{C_f\alpha_{\max}}{\vert\alpha\vert}.
        \qedhere
    \end{equation*} 
\end{proof}

We are now ready to state the main result of this section.
\begin{theorem}\label{theorem: non i.i.d}
    Let $(\sigma_n)_{n\ge 1}$ be positive numbers. Let $(g_n)_{n\ge 1}$ be simple graphs where $V(g_n)=[n]$ and consider the weighted graph $(g_n,\sigma^2)$, where $\sigma^2(i)=\sigma_i^2$. Assume that $(g_n,\sigma^2)_n$ converges to a graphon $w$, and that
    \begin{align*}
        &\vert \sigma^2\vert =\sum_{k\in [n]}\sigma_i^2 \to \infty; \quad \frac{\sigma_{\max}^2}{\vert \sigma^2\vert}\to 0.
    \end{align*}
    Let $a_1^{(n)},\ldots,a_n^{(n)} \in \A$ be $g_n$-independent centered random variables with variance $\var(a_k^{(n)})=\sigma_k^2$, for every $n \ge 1$. 
    Assume that there exist constants $C_p\ge 1$ that depends only on $p\ge 1$ such that for any $n,k,p\ge 1$, we have 
    \begin{align*}
       \tau\left(|a_k^{(n)}|^p\right)\le C_p\sigma_k^p.
    \end{align*}
    Then
    \begin{align*}
        S_n=\frac{1}{\sqrt{\vert\sigma^2\vert}}\sum_{k \in [n]} a_k^{(n)}
    \end{align*}
    converges in distribution to $\mu_w$, the law defined in \eqref{eq: definition of mu_w}.
\end{theorem}
\begin{proof}
   The proof follows the steps of that of Theorem \ref{theorem: 1}. For ease of notation, we will drop the superscript $n$ and write $a_k$ for $a_k^{(n)}$. We denote $z_k=a_k/\sigma_k$.
    For any integer $p\ge 1$, we have that
    \begin{align*}
        \tau(S_n^p)=\sum_{\pi \in P(p)}\frac{1}{\vert\sigma^2\vert^{p/2}}\sum_{\substack{i\in [n]^p\\ \ker (i)=\pi}}\sigma_{i_1}\cdots \sigma_{i_p}\tau(z_{i_1}\cdots z_{i_p}).
    \end{align*}
    As we have $\tau(z_{i_1}\cdots z_{i_p})=0$ for all partitions $\pi$ that have a block $V=\{v\}$, we get that
    \begin{align*}
        \tau(S_n^p)=\sum_{\pi \in P_{\ge 2}(p)}\frac{1}{\vert\sigma^2\vert^{p/2}}\sum_{\substack{i\in [n]^p\\ \ker (i)=\pi}}\sigma_{i_1}\cdots \sigma_{i_p}\tau(z_{i_1}\cdots z_{i_p}).
    \end{align*}
    Using the moment assumption on the $a_k$'s, we have $|\tau(z_{i_1}\cdots z_{i_p})|\le C_p$. Hence
    \begin{align*}
        \left|\sum_{\substack{i\in [n]^p\\ \ker (i)=\pi}}\sigma_{i_1}\cdots \sigma_{i_p}\tau(z_{i_1}\cdots z_{i_p})\right|\le C_p\sum_{\substack{i\in [n]^p\\ \ker (i)=\pi}}\sigma_{i_1}\cdots \sigma_{i_p}.
    \end{align*}
    We note that for any $k\ge 2$, we have
    \begin{align*}
        \sum_{i\in [n]}\sigma_i^k \le \sigma_{\max}^{k-2} \sum_{i\in [n]}\sigma_i^2=\sigma_{\max}^{k-2}\vert\sigma^2\vert.
    \end{align*}
    Hence, for any $\pi\in P_{\ge 2}(p)$, we have
    \begin{align*}
        \left|\sum_{\substack{i\in [n]^p\\ \ker (i)=\pi}}\sigma_{i_1}\cdots \sigma_{i_p}\tau(z_{i_1}\cdots z_{i_p})\right|&\le C_p\prod_{V\in \pi}\left(\sum_{i\in [n]}\sigma_i^{|V|}\right)\\
        &\le C_p\prod_{V\in \pi}\left(\sigma_{\max}^{|V|-2}\vert\sigma^2\vert\right).
    \end{align*}
    Since $\sum_{V\in \pi}|V|=p$, we deduce that
    \begin{align*}
        \left|\frac{1}{\vert\sigma^2\vert^{p/2}}\sum_{\substack{i\in [n]^p\\ \ker (i)=\pi}}\sigma_{i_1}\cdots \sigma_{i_p}\tau(z_{i_1}\cdots z_{i_p})\right|\le C_{p}\left(\frac{\sigma_{\max}^2}{\vert\sigma^2\vert}\right)^{p/2-|\pi|}.
    \end{align*}
    Let $\pi \in P_{\ge 2}(p)$. If there exists $V\in \pi$ such that $|V|\ge 3$, we get that $|\pi|-p/2<0$. Since  $\sigma_{\max}^2/\vert\sigma^2\vert$ converges to zero as $n$ goes to infinity, we get in that case 
    \begin{align*}
        \frac{1}{\vert\sigma^2\vert^{p/2}}\sum_{\substack{i\in [n]^p\\ \ker (i)=\pi}}\sigma_{i_1}\cdots \sigma_{i_p}\tau(z_{i_1}\cdots z_{i_p}) \to 0.
    \end{align*}
    We deduce that
    \begin{align*}
        \lim_{n\to \infty}\tau(S_n^p)=\sum_{\pi \in P_2(p)}\lim_{n\to \infty}\frac{1}{\vert\sigma^2\vert^{p/2}}\sum_{\substack{i\in [n]^p\\ \ker (i)=\pi}}\sigma_{i_1}\cdots \sigma_{i_p}\tau(z_{i_1}\cdots z_{i_p}).
    \end{align*}
    The variables $z_k$ are centered with variance one, hence
    \begin{align*}
        \tau(z_{i_1}\cdots z_{i_p})=\mathbf{1}_{\pi \in \NC_2(g_n,i)}.
    \end{align*}
    In particular, it is equal to one if and only if the map $V=\{u,v\}\mapsto i_{u}=i_v$ is an (injective) homomorphism from the graph $f_\pi$ to the graph $g_n$. Hence
    \begin{align*}
        \sum_{\substack{i\in [n]^p\\ \ker (i)=\pi}}\sigma_{i_1}\cdots \sigma_{i_p}\tau(z_{i_1}\cdots z_{i_p})=\sum_{\phi \in \Homo_{\operatorname{inj}}(f_\pi,g_n)}\prod_{v\in V(f_\pi)}\sigma_{\phi(v)}^2.
    \end{align*}
  In view of \eqref{eq: def inj-hom-weight}, we can write
    \begin{align*}
        \lim_{n\to \infty}\tau(S_n^p)=\sum_{\pi \in P_2(p)}\lim_{n\to \infty}\rho_{\operatorname{inj}}(f_\pi,(g_n,\sigma^2)).
    \end{align*}
    Following Lemma~\ref{lemma: diff between homo and homoinj}, the homomorphism densities $\rho_{\operatorname{inj}}(f,(g_n,\sigma^2))$ and $\rho(f,(g_n,\sigma^2))$ are asymptotically indistinguishable since $\sigma_{\max}^2/\vert\sigma^2\vert$ goes to zero. As $(g_n,\sigma^2) \to w$, we deduce that
    \begin{align*}
        \lim_{n\to \infty}\tau(S_n^p)=\sum_{\pi \in P_2(p)}\rho(f_\pi,w),
    \end{align*}
    and the result follows.
\end{proof}
\begin{remark}
    In classical and free probability, the condition $\sigma_{\max}^2/|\sigma^2|$ is related to the Liapunov condition: for some $\delta >0$, we have
    \begin{align*}
        \frac{1}{|\sigma^2|^{1+\delta/2}}\sum_{k\in [n]}\tau(|a_k|^{2+\delta})\to 0.
    \end{align*}
    In particular, if a sequence of independent centered random variables $a_k$ satisfies the Liapunov condition for some $\delta>0$, Holder's inequality implies that
    \begin{align*}
        \frac{\sigma_{\max}^{2(1+\delta/2)}}{|\sigma^2|^{1+\delta/2}}\le \frac{1}{|\sigma^2|^{1+\delta/2}}\sum_{k\in [n]}\tau\left(|a_k|^{2+\delta}\right)\to 0.
    \end{align*}
    Therefore, Theorem \ref{theorem: non i.i.d} implies Liapunov central limit theorem under the uniform bound on the variables $a_k$ for classical independent random variables ($w_{g,\sigma^2}\equiv 1$), and its free analog ($w_{g,\sigma^2}\equiv 0$); see \cite{chistyakov2008}.
\end{remark}

\section{Multivariate case}\label{sec: multivariate case}
The goal of this section is to extend Theorem \ref{theorem: 1} to the multivariate case; namely, we study the joint convergence of
\begin{align*}
    S_n^{(l)}=\frac{1}{\sqrt{n}}\sum_{k\in [n]}a_k^{(l)},
\end{align*}
where $(a_k^{(l)})_{k\in [n],l\in [L]}$ are $g_n$-independent identically distributed random variables according to some graph $g_n$ over $[n]\times [L]$. We will make use of the notion of decorated graphons from \cite{lovasz2022multigraph} in order to adapt the graphon convergence to the above setting. 

Let $L\ge 1$ and $\mathcal{M}_L=\mathcal{M}_L([0,1])$ be the space of $[0,1]$-valued $L\times L$ matrices. We denote the inner product  and Hilbert-Schmidt norm by 
\begin{align*}
    &\langle S,T\rangle=\tr(ST^*)=\sum_{i,j\in [L]}S_{ij}T_{ij};\\
    &\norm{S}_{\operatorname{HS}}^2=\sum_{i,j\in [L]}S_{ij}^2.
\end{align*}
We define an $\mathcal{M}_L$-decorated graphon $w$ as a map $w:[0,1]^2\to \mathcal{M}_L$ satisfying the symmetry condition $\norm{w(x,y)}_{\operatorname{HS}}=\norm{w(y,x)}_{\operatorname{HS}}$ for every $x,y\in [0,1]$.

An $\mathcal{M}_L$-decorated graph $G=(g,\beta)$ is a pair consisting of a graph $g=(V(g),E(g))$ and a map $\beta:E(g)\to \mathcal{M}_L$. 
Given an $\mathcal{M}_L$-decorated graph $G=(g,\beta)$, one can naturally associate an $\mathcal{M}_L$-decorated graphon $w_G$ defined by
\begin{align*}
    &w_G(x,y)=\beta(i,j);\\
    &(x,y)\in \left[\frac{i-1}{n},\frac{i}{n}\right)\times \left[\frac{j-1}{n},\frac{j}{n}\right).
\end{align*}
Given an $\mathcal{M}_L$-decorated graph $F=(f,\beta)$ and an $\mathcal{M}_L$-decorated graphon $w$, we define
\begin{align}
    \rho(F,w)=\int_{[0,1]^{V(f)}}\prod_{e=(u,v)\in E(f)}\langle \beta(e),W(x_u,x_v)\rangle\,\text{d}x. \label{eq:rhomatrices}
\end{align}
When the graphon $w$ is generated by an $\mathcal{M}_L$-decorated graph $G$ i.e. $w=w_G$, the above definition reduces to 
\begin{align*}
    &\rho(F,G):=\rho(F,w_G)=\frac{\hom(F,G)}{|V(g)|^{|V(f)|}};\\
    &\hom(F,G)=\sum_{\phi:V(f)\to V(g)}\prod_{e\in E(f)}\langle \beta(e),\gamma(\phi(e))\rangle,
\end{align*}
where $\phi(e)=(\phi(u),\phi(v))$, for $e=(u,v)$. 
Similarly, we define
\begin{align*}
    &\homo_{\operatorname{inj}}(F,G)=\sum_{\phi:V(f)\hookrightarrow V(g)}\prod_{e\in E(f)}\langle \beta(e),\gamma(\phi(e))\rangle;\\
    &\rho_{\operatorname{inj}}(F,G)=\frac{\homo_{\operatorname{inj}}(F,G)}{|V(g)|^{|V(f)|}},
\end{align*}
where $\phi: V(f) \hookrightarrow V(g)$ denotes an injective map. 

For a graph $f=(V(f),E(f))$ and $l\in [L]^{V(f)}$, we define its $\mathcal{M}_L$-decorated graph $F_f^l=(f,\beta_l)$ given by
\begin{align*}
    \beta_l(u,v)=P_{l_u,l_v},
\end{align*}
where $(u,v)\in E(f)$ and $P_{i,j}$ is the canonical basis in $\mathcal{M}_L$.
For a graph $g_n=([n]\times [L],E(g_n))$, we define an $\mathcal{M}_L$-decorated graph $G_{g_n}=(V(G_{g_n}),E(G_{g_n}),\gamma)$ as follows. The graph $(V(G_{g_n}),E(G_{g_n}))$ is the complete graph $K_n$ over $n$ vertices, and for any $u,v\in [n]$, we define
\begin{align*}
    \gamma(u,v)=\left(\mathbf{1}_{((u,i),(v,j))\in E(g_n)}\right)_{i,j\in [L]}.
\end{align*}
Note that $\gamma(u,v)=\gamma(v,u)^*$. Informally, for any $i,j\in [L]$, the matrix $(\gamma_{ij}(u,v))_{u,v\in [n]}$ is a ``two-layer'' adjacency matrix of $g_n$, namely, $\gamma_{ij}(u,v)=1$ if and only if there exists an edge from $(u,i)$ to $(v,j)$. 
Since 
\begin{align*}
    \norm{\gamma(u,v)}_{\operatorname{HS}}=\norm{\gamma(u,v)^*}_{\operatorname{HS}}=\norm{\gamma(v,u)}_{\operatorname{HS}},
\end{align*}
the graphon $w_{G_{g_n}}$ is well-defined as an $\mathcal{M}_L$-decorated graphon. 
It was proved in \cite[Theorem 3.7]{lovasz2022multigraph} that, given a sequence of $\mathcal{M}_L$-decorated graphons $(w_n)_{n\geq 1}$,  if $\rho(F,w_n)$ converges as $n\to \infty$ for every $\mathcal{M}_L$-decorated graph $F$, then there exists an $\mathcal{M}_L$-decorated graphon $w$ such that $\rho(F,w_n)\to \rho(F,w)$, for every $\mathcal{M}_L$-decorated graph $F$. We denote this convergence by $w_n\to w$. Let us note that it is possible to define the cut distance $\delta_\Box$ on the space of $\mathcal{M}_L$-decorated graphons in such a way that, for any simple graph $F$, the map $w\mapsto \rho(F,w)$ is continuous for $\delta_\Box$. More precisely, set
$$\|w\|_{\Box}:=\sup_{f,g:[0,1]\to [0,1]}\left\|\int_{[0,1]^2}w(x,y) f(x)g(y)\text{d}x\text{d}y\right\|_{HS}.$$
Then,
\begin{align*}
    \delta_{\Box}(w_1,w_2):=\inf_{\phi}\norm{w_1-w_2\circ \phi}_{\Box},
\end{align*}
where the infimum is over measure-preserving transformations $\phi:[0,1]\to [0,1]$ and
\begin{align*}
    w\circ\phi (x,y):=w(\phi(x),\phi(y)).
    \end{align*}
This distance $\delta_{\Box}$ gives a criterion for the convergence $w_n\to w$. The proof of such continuity is carried out in \cite[Section 2]{lovasz2022multigraph}.

We can now state the main result of this section.
\begin{theorem}\label{theorem: multivariate case}
    Let $L\ge 1,n\ge 1$ and $g_n=([n]\times [L],E(g_n))$ be a grid graph. Assume that $w_{G_{g_n}}\to w$, where $w$ is an $\mathcal{M}_L$-decorated graphon. Let $a\in \A$ be a normalized random variable and let $(a_k^{(r)})_{k\in [n],r\in [L]}$ be $g_n$-independent and identically distributed random variable with common law $a$. Let
    \begin{align*}
        S_n^{(r)}=\frac{1}{\sqrt{n}}\sum_{k\in [n]}a_k^{(r)}.
    \end{align*}
    Then, $(S_n^{(r)})_{r\in [L]}$ converges in distribution to a family $(s_r)_{r\in [L]}$ whose joint distribution $\nu_{w}$ depends only on $w$. Moreover, for any $p\ge 1$ and $l\in [L]^p$, we have
    \begin{align*}
        \tau(s_{l_1}\cdots s_{l_p})=\sum_{\substack{\pi \in P_2(p)\\\pi\le \ker (l)}}\rho(F_{f_\pi}^l,w)=:\nu_w(l).
    \end{align*}
\end{theorem}
One particular instance of Theorem \ref{theorem: multivariate case} is the case of a lexicographical product graph. Let $g_L$ be a graph over $[L]$, and let $g'_n$ be a graph over $[n]$. We consider the lexicographical product graph $g_{n,L}=g_L\cdot g'_n$, whose vertex set is $[n]\times [L]$ and edge set
\begin{align*}
    E(g_{n,L})=\{((u,i),(v,i)): (u,v)\in E(g'_n))\}\sqcup\{((u,i),(v,j)): (i,j)\in E(g_L)\},
\end{align*}
see Figure \ref{fig: lexicographic product} for an example $n=L=3$.
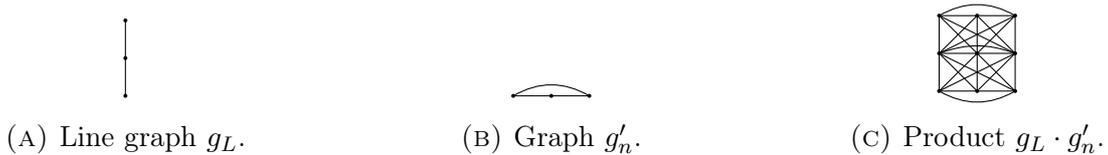
\begin{figure}[H]
     \centering
     \begin{subfigure}[b]{0.3\textwidth}
         \centering
         \begin{tikzpicture}[scale=0.50]
            \node at (0,0) (A) {};
            \fill[black] (A) circle(1.5pt);
            \node at (0,1) (B) {};
            \fill[black] (B) circle(1.5pt);
            \node at (0,2) (C) {};
            \fill[black] (C) circle(1.5pt);

            \draw[-] (0,0) edge (0,1);
            \draw[-] (0,1) edge (0,2);
        \end{tikzpicture}
        \caption{Line graph $g_L$.}
    \end{subfigure}
     \hfill
     \begin{subfigure}[b]{0.3\textwidth}
         \centering
         \begin{tikzpicture}[scale=0.50]
            \node at (0,0) (A) {};
            \fill[black] (A) circle(1.5pt);
            \node at (1,0) (B) {};
            \fill[black] (B) circle(1.5pt);
            \node at (2,0) (C) {};
            \fill[black] (C) circle(1.5pt);

            \draw[-] (0,0) edge (1,0);
            \draw[-] (1,0) edge (2,0);
            \draw[-] (0,0) edge [bend left=30] (2,0);
        \end{tikzpicture}
         \caption{Graph $g'_n$.}
     \end{subfigure}
     \hfill
     \begin{subfigure}[b]{0.3\textwidth}
         \centering
         \begin{tikzpicture}[scale=0.50]
            \node at (0,0) (A) {};
            \fill[black] (A) circle(1.5pt);
            \node at (0,1) (B) {};
            \fill[black] (B) circle(1.5pt);
            \node at (0,2) (C) {};
            \fill[black] (C) circle(1.5pt);

            \node at (1,0) (A) {};
            \fill[black] (A) circle(1.5pt);
            \node at (1,1) (B) {};
            \fill[black] (B) circle(1.5pt);
            \node at (1,2) (C) {};
            \fill[black] (C) circle(1.5pt);

            \node at (2,0) (A) {};
            \fill[black] (A) circle(1.5pt);
            \node at (2,1) (B) {};
            \fill[black] (B) circle(1.5pt);
            \node at (2,2) (C) {};
            \fill[black] (C) circle(1.5pt);

            \foreach \x in {0,1,2}
            \foreach \y in {0,1,2}
              {
                \draw[-] (\x,0) edge (\y,1);
                \draw[-] (\x,1) edge (\y,2);
              }
            \draw[-] (0,0) edge (1,0);
            \draw[-] (1,0) edge (2,0);
            \draw[-] (0,0) edge [bend right=30] (2,0);        

            \draw[-] (0,1) edge (1,1);
            \draw[-] (1,1) edge (2,1);
            \draw[-] (0,1) edge [bend left=20] (2,1);   

            \draw[-] (0,2) edge (1,2);
            \draw[-] (1,2) edge (2,2);
            \draw[-] (0,2) edge [bend left=30] (2,2);   
        \end{tikzpicture}
         \caption{Product $g_L\cdot g'_n$.}
     \end{subfigure}
        \caption{Example of lexicographical product.}
        \label{fig: lexicographic product}
\end{figure}
In this case, the variables $(a_k^{(l)})_{k\in [n]}\subset \A_l$ are $g'_n$-independent, and the subalgebras $(\A_l)_{l\in [L]}$ are $g_L$-independent. Specifying to this example, Theorem \ref{theorem: multivariate case} yields the following.
\begin{corollary}
    Let $g_L$ be a graph over $[L]$, $g'_n$ be a graph over $[n]$, and $a\in \A$ with mean zero and variance one. Assume that $g'_n\to w$. Let $(\A_l)_{l\in [L]}$ be $g_L$-independent subalgebras and $(a_k^{(l)})_{k\in [n]}\subset \A_l$ be $g'_n$-independent and identically distributed random variables with common law $a$. Let
    \begin{align*}
        S_n^{(l)}=\frac{1}{\sqrt{n}}\sum_{k\in [n]}a_k^{(l)}.
    \end{align*}
    Then, $(S_n^{(l)})_{l\in [L]}$ converges to a joint law $\nu_w$ that depends only on $g_L$ and $w$. Moreover, the law $\nu_w$ can be written as the law of $(s_l)_{l\in [L]}$, where each $s_l$ has law $\mu_{w}$ and they are $g_L$-independent.
\end{corollary}
\begin{proof}
    Let $g_{n,L}=g_l\cdot g'_n$. It is immediate to check that $(a_k^{(l)})_{k\in [n],l\in [L]}$ are $g_{n,L}$-independent.
    Let $G_{g_{n,L}}=(K_n,\gamma)$. We note that
    \begin{align*}
        \gamma_{ij}(u,v)=\begin{cases}
            \mathbf{1}_{(u,v)\in E(g'_n)};&\text{ if }i=j;\\
            \mathbf{1}_{(i,j)\in E(g_L)};& \text{ otherwise}.
        \end{cases}
    \end{align*}
    By the cut-distance metric, it is easy to see that $w^{(n)}=w_{G_{g_{n,L}}}$ converges to 
    \begin{align*}
        w^*_{ij}=\begin{cases}
            w;&\text{ if }i=j;\\
            \mathbf{1}_{(i,j)\in E(g_L)};& \text{ otherwise}.
        \end{cases}
    \end{align*}
    Theorem \ref{theorem: multivariate case} implies that the joint law of $(S_n^{(l)})_{l\in [L]}$ converges towards $\nu_{w^*}$. By fixing $k\in [L]$ and $l=(k,\ldots,k)\in [L]^p$, we can see that
    \begin{align*}
        \nu_{w^*}(l)=\sum_{\pi\in P_2(p)}\int_{[0,1]^{V(f_\pi)}}\prod_{(u,v)\in E(f_\pi)} w_{kk}^*(x_u,x_v)\,\text{d}x.
    \end{align*}
    Hence, the law of each $s_k$ is given by $\mu_{w}$ as $w_{kk}^*=w$. Moreover, since $S_n^{(l)}$ belongs to $\A_l$, where $\A_l$ are $g_L$-independent subalgebras, Lemma \ref{lemma: speicher} implies that
    \begin{align*}
        \tau(S_n^{(l_1)}\cdots S_n^{(l_p)})=\sum_{\sigma \in \NC(g_L,l)}\prod_{V=\{v_1,\ldots,v_m\}\in\sigma}\cumuf_{m}(S_n^{(l_{v_1})}),
    \end{align*}
    where $m=|V|$.
    By Theorem \ref{theorem: 1} and the moment-cumulant formula, we get that $\cumuf_{m}(S_n^{(l)})\to \cumuf_{m}(s_l)$, for every $m\ge 1$. Hence
    \begin{align*}
        \tau(S_n^{(l_1)}\cdots S_n^{(l_p)})\to \sum_{\sigma \in \NC(g_L,l)}\prod_{V=\{v_1,\ldots,v_m\}\in\sigma}\cumuf_{m}(s_{l_{v_1}}).
    \end{align*}
    We deduce by Theorem \ref{theorem: multivariate case} that
    \begin{align*}
        \nu_w(l)=\lim_{n\to \infty}\tau(S_n^{(l_1)}\cdots S_n^{(l_p)})=\tau(s_{l_1}\cdots s_{l_p}),
    \end{align*}
    where $(s_l)_{l\in [L]}$ are $g_L$-independent random variables. The proof is complete.
\end{proof}

\begin{proof}[Proof of Theorem \ref{theorem: multivariate case}]
    Let $l\in [L]^p$. We aim to show that
    \begin{align*}
        \tau\left(S_n^{(l_1)}\cdots S_n^{(l_p)}\right)\to \sum_{\substack{\pi \in P_2(p)\\\pi\le \ker(l)}}\rho(F_{f_\pi}^l,w).
    \end{align*}
    By standard reductions, we can see that 
    \begin{align*}
        \lim_{n\to \infty}\tau\left(S_n^{(l_1)}\cdots S_n^{(l_p)}\right)=\sum_{\pi \in P_2(p)}\lim_{n\to \infty}\frac{1}{n^{p/2}}\sum_{\substack{i\in [n]^p\\\ker (i)=\pi}}\tau\left(a_{i_1}^{(l_1)}\cdots a_{i_p}^{(l_p)}\right).
    \end{align*}
    Let $\theta=((i_1,l_1),\ldots,(i_p,l_p))$. By Lemma \ref{lemma: speicher}, we get that
    \begin{align*}
        \tau\left(a_{i_1}^{(l_1)}\cdots a_{i_p}^{(l_p)}\right)=\sum_{\sigma \in \NC(g_n,\theta)}\cumuf_\sigma(a_{i_1}^{(l_1)},\ldots ,a_{i_p}^{(l_p)}).
    \end{align*}
    As the variables are centered, we get
    \begin{align*}
        \tau\left(a_{i_1}^{(l_1)}\cdots a_{i_p}^{(l_p)}\right)=\sum_{\sigma \in \NC_{\ge 2}(g_n,\theta)}\cumuf_\sigma(a_{i_1}^{(l_1)},\ldots ,a_{i_p}^{(l_p)}).
    \end{align*}
    Let $\sigma \in \NC_{\ge 2}(g_n,\theta)$. Any block $V=\{v_1,\ldots,v_m\}\in \sigma$ must satisfy $i_{v_1}=\cdots =i_{v_m}$ by definition of $\sigma \le \ker (\theta)$. In particular, we get that $m=2$ as $\pi$ is a pair partition and $\ker (i)=\pi$. In addition, $V=\{v_1,v_2\}\in \sigma$ if and only if $i_{v_1}=i_{v_2}$, that is, $\sigma=\ker (i)=\pi$. As $\var(a)=1$, we get that
    \begin{align*}
        \tau\left(a_{i_1}^{(l_1)}\cdots a_{i_p}^{(l_p)}\right)=\mathbf{1}_{\pi \in \NC_2(g_n,\theta)}.
    \end{align*}
    The following is equivalent to $\pi\in \NC_2(g_n,\theta)$.
    \begin{enumerate}
        \item For any block $V=\{v_1,v_2\}\in \pi$, we have $l_{v_1}=l_{v_2}$ (as $\pi\le \ker(\theta)$); and
        \item For any edge $(\{u_1,u_2\},\{v_1,v_2\})\in E(f_\pi)$, we have $((i_{u_1},l_{u_1}),(i_{v_1},l_{v_1}))\in E(g_n)$.
    \end{enumerate}
    For condition (1), we see that $\pi$ must satisfy $\pi\le \ker (l)$. Hence, for any $\pi\in P_2(p)$ such that $\pi\le \ker (l)$, we have
    \begin{align*}
        \sum_{\substack{i\in [n]^p\\\ker (i)=\pi}}\tau\left(a_{i_1}^{(l_1)}\cdots a_{i_p}^{(l_p)}\right)&=\sum_{\substack{i\in [n]^p\\\ker (i)=\pi}}\mathbf{1}_{\pi\in \NC_2(g_n,\theta)}\\
        &=\sum_{\phi: V(f_\pi)\hookrightarrow [n]} \prod_{e=(u,v)\in E(f_\pi)}\mathbf{1}_{((\phi(u),l_{u}),(\phi(v),l_{v}))\in E(g_n)},
    \end{align*}
    where $u=\{u_1,u_2\}\in \pi$ and $l_u:=l_{u_1}=l_{u_2}$ (similarly for $v\in \pi$).
    Recall the definitions of the $\mathcal{M}_L$-decorated graphs $F^{l}_{f_\pi}=(f_\pi,\beta_l)$ and $G_{g_n}=(K_n,\gamma)$. We aim to show that for any $e=(u,v)\in E(f_\pi)$, we have
    \begin{align*}
        \mathbf{1}_{((\phi(u),l_{u}),(\phi(v),l_{v}))\in E(g_n)}=\langle \beta_l(e),\gamma(\phi(e))\rangle.
    \end{align*}
    Indeed, we first have that
    \begin{align*}
        \beta_l(u,v)=P_{l_u,l_v}.
    \end{align*}
    The $l_ul_v$-entry of $\gamma(\phi(u),\phi(v))$ is given by
    \begin{align*}
        \gamma_{l_u,l_v}(\phi(u),\phi(v))=\mathbf{1}_{((\phi(u),l_u),(\phi(v),l_v))\in E(g_n)}.
    \end{align*}
    Hence
    \begin{align*}
        \langle \beta_l(e),\gamma(\phi(e))\rangle=\mathbf{1}_{((\phi(u),l_u),(\phi(v),l_v))\in E(g_n)}.
    \end{align*}
    We deduce that
    \begin{align*}
        \sum_{\substack{i\in [n]^p\\\ker (i)=\pi}}\tau\left(a_{i_1}^{(l_1)}\cdots a_{i_p}^{(l_p)}\right)=\homo_{\operatorname{inj}}(F^l_{f_\pi},G_{g_n}).
    \end{align*}
    Hence
    \begin{align*}
        \lim_{n\to \infty}\tau\left(S_n^{(l_1)}\cdots S_n^{(l_p)}\right)=\sum_{\substack{\pi \in P_2(p)\\\pi\le \ker (l)}}\lim_{n\to \infty}\rho_{\operatorname{inj}}(F^l_{f_\pi},G_{g_n}).
    \end{align*}
    By \cite[Lemma 2.3]{lovasz2022multigraph} and a similar argument done in Lemma \ref{lemma: diff between homo and homoinj}, we can bound
    \begin{align*}
        |\rho(F^l_{f_\pi},G_{g_n})-\rho_{\operatorname{inj}}(F^l_{f_\pi},G_{g_n})|\le \frac{C_f}{n}.
    \end{align*}
    Therefore, we get that $\rho_{\operatorname{inj}}(F^l_{f_\pi},G_{g_n})$ and $\rho(F^l_{f_\pi},G_{g_n})$ are indistinguishable when $n$ goes to infinity. Hence, the graphon convergence $w_{G_{g_n}}\to w$ implies that
    \begin{align*}
        \lim_{n\to \infty}\tau\left(S_n^{(l_1)}\cdots S_n^{(l_p)}\right)=\sum_{\substack{\pi \in P_2(p)\\\pi\le \ker (l)}}\rho(F^l_{f_\pi},w),
    \end{align*}
    and the result follows.
\end{proof}

\section{Graphon with value in $[-1,1]$}\label{sec: negative graphon}
In this section, we state a version of Theorem \ref{theorem: 1} where the limiting measure can be any $\mu_{w}$ with $w:[0,1]^2\to [-1,1]$, removing the restriction over the range of $w$. To this aim, we need to go beyond $\epsilon$-independence, and consider variables which are "$Q$-independent", for a symmetric matrix
$Q=(q_{ij})^n_{i,j=1}$ (the case $q_{ij} \in \{0,1\}$ being the particular case of $\epsilon$-independence). Even without a general notion of "$Q$-independence" (see~\cite{van1996obstruction}), there is a good way to consider $q$-Gaussian variables which behave as if they were "$Q$-independent": it is the notion of mixed $q$-Gaussian system, introduced by Speicher in~\cite{speicher1993generalized}. Given a symmetric matrix
$Q=(q_{ij})^n_{i,j=1}$ with $q_{ij} \in [-1, 1]$, one considers variables satisfying
$$a_i a^{*}_j-q_{ij}a^{*}_j a_i = \delta_{ij}.$$
The mixed $q$-Gaussian system is a vector of self-adjoint variables $s_j = a^{*}_
j + a_j$, whose noncommutative distribution is given by
 $$\tau[s_{i_1}\cdots s_{i_p}]= \sum_{\substack{\pi \in P_2(p)\\ \pi \leq \ker(i)}}\prod_{(u,v)\in E(f_{\pi})}q_{i_u,i_v},$$
where $f_{\pi}$ is the intersection graph of $\pi$ (see \cite{speicher1993generalized} for the introduction of the mixed $q$-Gaussian system). The graphon $w_Q:[0,1]^2\to [-1,1]$ generated by the matrix $Q$ is defined as follows: for $i,j \in [n]$ and
\begin{align*}
    x \in \left[\frac{i-1}{n},\frac{i}{n}\right),\quad y \in \left[\frac{j-1}{n},\frac{j}{n}\right),
\end{align*}
we set $w_Q(x,y)=q_{ij}$. With this definition, it follows that, for any graph $f=(V(f),E(f))$, we have 
\begin{align*}
    \rho(f,w_Q)=\sum_{i\in [n]^{V(f)}}\frac{1}{n^{|V(f)|}}\prod_{(u,v)\in E(f)}q_{j_u,j_v}.
\end{align*}
We deduce the following version of the CLT.

\begin{theorem}\label{theorem: clt for negative graphons}
    Let $(Q_n)_{n \ge 1}$ be a sequence of symmetric $n\times n$ matrices with entries in $[-1,1]$. Assume that $w_{Q_n} \to w$ in the graphon sense. Let $s_1,\ldots,s_n \in \A$ be a mixed $q$-Gaussian system, for every $n \ge 1$. Then 
    \begin{align*}
        S_n=\frac{1}{\sqrt{n}}\sum_{k \in [n]}s_k
    \end{align*}
    converges in distribution to a law $\mu_w$ that depends only on $w$. Moreover, its odd moments vanish, whereas, for any $p\in \N$, we have
    \begin{align*}
        \int x^{2p}\, \text{d}\mu_w=\sum_{\pi \in P_2(2p)}\rho(f_\pi,w).
    \end{align*}
\end{theorem}
Note that in \cite[Theorem 2]{speicher1993generalized}, Speicher considered a sequence of symmetric $n\times n$ matrices with entries in $\{-1,1\}$ with independent and identically distributed variables (up to the symmetry). In this case, the law of large numbers implies that $w_{Q_n}$ converges almost surely in the graphon sense to a constant graphon, which allows us to recover \cite[Theorem 2]{speicher1993generalized} by using Theorem~\ref{theorem: clt for negative graphons}.
\begin{proof}
    Let $p\in \mathbb{N}$. We compute
    \begin{align*}
        \tau(S_n^p)=\frac{1}{n^{p/2}}\sum_{i\in [n]^p} \sum_{\pi \in P_2(p)}\mathbf{1}_{\pi \le \ker (i)} \prod_{(u,v)\in E(f_\pi)} q_{i_u,i_v}.
    \end{align*}
    Note that the set of $i\in [n]^p$ such that $\pi\le \ker (i)$ is in a one-to-one with the set of all maps $\phi: V(f_\pi)\to [n]$. Hence, by definition of $\rho(f,{w_{Q_n}})$, we get
    \begin{align*}
        \tau(S_n^p)=\sum_{\pi\in P_2(p)} \rho(f_\pi,w_{Q_n}).
    \end{align*}
    The result follows by the graphon convergence $w_{Q_n} \to w$.
\end{proof}

\section{Examples}\label{sec: examples}

The goal of this section is to provide concrete examples illustrating the wide range of applicability of Theorem \ref{theorem: 1}, as well as the variety of limiting measures appearing in this central limit theorem. 
The free central limit theorem, when the graph of independence is edgeless, stipulates that the semi-circle law appears as the limit. 
We start by checking that the limiting behavior is unaffected as long as the graphs of independence are sparse enough. This follows from the universality of Theorem~\ref{theorem: 1} and the fact that the graphon limit of sparse graphs coincides with that of edgeless graphs. 
\begin{example}\label{ex: semi-circle}
    Let $g_n$ be a sequence of graphs and suppose that the maximum degree $\Delta_n$ of $g_n$ grows sublinearly, i.e.,
    \begin{align*}
        \lim_{n\to \infty} \frac{\Delta_n}{n}=0.
    \end{align*}
    Let $a_1,\ldots, a_n \in \A$ be $g_n$-independent identically distributed random variables with mean zero and variance one. Then 
    \begin{align*}
        S_n=\frac{1}{\sqrt{n}}\sum_{k \in [n]}a_k
    \end{align*}
    converges in distribution to the semi-circle law.
\end{example}
\begin{proof}
    Note first that if $f$ is a connected graph with $k$ vertices, then
    \begin{align*}
        \homo(f,g_n) \le n\Delta_n^{k-1}.
    \end{align*}
    We then have
    \begin{align*}
        \rho(f,g_n)=\frac{\homo(f,g_n)}{n^{k}} \le \left(\frac{\Delta_n}{n}\right)^{k-1}
    \end{align*}
    By assumption, whenever $f$ has more than one vertex, we have $\rho(f,g_n) \to 0$. Therefore, $g_n$ converges to the graphon $w \equiv 0$. Now, let $\pi \in P_2(p)$ be a pair partition. Denote $\pi_1,\ldots,\pi_k$ its connected components. Since $\rho(f,g_n)$ is a multiplicative function, i.e.,
    \begin{align*}
        \rho(f_\pi,g_n)=\prod_{l \in [k]}\rho(f_{\pi_l},g_n),
    \end{align*}
    we have that whenever $\pi$ has an edge, $\rho(f_\pi,g_n) \to 0$. In other words, only noncrossing partitions contribute. For $\pi \in NC_2(p)$, we get $\rho(f_\pi,g_n)=1$ and the result follows.
\end{proof}

The classical central limit theorem, when the graph of independence is complete, stipulates that the Gaussian law appears as the limit. As mentioned above, this limiting behavior is unaffected as long as the graphs of independence are sufficiently dense.

\begin{example}\label{ex: gaussian}
  Let $g_n$ be a sequence of graphs and suppose that the minimum degree $\delta_n$ of $g_n$ satisfies
    \begin{align*}
       \lim_{n\to \infty}\frac{\delta_n}{n} =1 
    \end{align*}
    Let $a_1,\ldots, a_n \in \A$ be $g_n$-independent identically distributed random variables with mean zero and variance one. Then 
    \begin{align*}
        S_n=\frac{1}{\sqrt{n}}\sum_{k \in [n]}a_k
    \end{align*}
    converges in distribution to the Gaussian law.
\end{example}
\begin{proof}
    Let $f$ be a connected graph with $k$ vertices. Then, the definition of the minimum degree $\delta_n$ implies that
    \begin{align*}
        \hom(f,g_n) \ge n\delta_n^{k-1}.
    \end{align*}
    We deduce that
    \begin{align*}
        1\ge \rho(f,g_n) \ge \left(\frac{\delta_n}{n}\right)^{k-1} \to 1.
    \end{align*}
    By multiplicativity, $\rho(f,g_n) \to 1$ for all simple graphs $f$. Hence, $g_n\to w$, where $w\equiv 1$ is the constant graphon. We can easily see that
    \begin{align*}
        \int x^{2p}\,\text{d}\mu_w=\sum_{\pi\in P_2(2p)}1=|P_2(2p)|,
    \end{align*}
    so $\mu_w$ is the Gaussian law, and the proof is complete.
\end{proof}

Theorem~\ref{theorem: 1} also allows us to capture the $q$-Gaussian laws \eqref{eq: q-gaussian def}, recovering a result in~\cite{mlotkowski2004lambdafree}.
\begin{example}\label{example: q-Gaussian}
    Let $q \in (0,1)$. Let $g_n$ be a sequence of graphs converging to the graphon $w(x,y)=q$, for all $x \ne y$. Let $a_1,\ldots, a_n \in \A$ be $g_n$-independent identically distributed random variables with mean zero and variance one. Then 
    \begin{align*}
        S_n=\frac{1}{\sqrt{n}}\sum_{k \in [n]}a_k
    \end{align*}
    converges to the $q$-Gaussian law.  
\end{example}
\begin{proof}
This follows automatically since  by the definition of $\rho(f,w)$, for all graphs $f$, we have $
        \rho(f,w)=q^{|E(f)|}.$
\end{proof}

One particular feature of Theorem \ref{theorem: 1} is that it captures some exotic distributions as a universal limit in a central limit theorem for some $\epsilon$-independent random variables. Let $+$ be the classical convolution and $\boxplus$ be the free convolution. Let $\mu_{sc},\mu_g$ be the standard semi-circle and Gaussian law, respectively, and let $t\mu$ denote the dilation of $\mu$ by $t$, 
\begin{align*}
    t\mu(A):=\mu(t^{-1}A),
\end{align*}
for all measurable sets $A$.

\begin{example}
The following laws 
  $$
\text{(1) $\frac{1}{\sqrt{2}}(\mu_{sc}+\mu_{sc})$; \quad (2) $\frac{1}{\sqrt{2}} (\mu_g+\mu_{sc})$;\quad (3) $\frac{1}{\sqrt{2}}(\mu_g\boxplus\mu_g)$;\quad (4) $\frac{1}{\sqrt{2}}(\mu_g\boxplus\mu_{sc})$}
$$ 
coincide with $\mu_w$ where, following the same order, $w$ is the graphon given by 
\begin{align*}
   &\text{(1) } \mathbf{1}_{[0,1/2)\times [1/2,1]}+\mathbf{1}_{[1/2,1]\times[0,1/2)}; \quad 
   \text{(2) } 1-\mathbf{1}_{[1/2,1]\times [1/2,1]}\\
   &\text{(3) } \mathbf{1}_{[0,1/2)\times [0,1/2)}+\mathbf{1}_{[1/2,1]\times [1/2,1]};\quad 
   \text{(4) } \mathbf{1}_{[0,1/2)\times [0,1/2)}.
\end{align*}
\end{example}
\begin{proof}
    This follows by Proposition \ref{proposition: mu_w is closed under cla/free} and the construction in Remark \ref{remark: construction convolution}.
\end{proof}
To conclude, we give some nontrivial measures associated with negative graphons.
\begin{example}The following laws
\begin{align*}
   &\text{(1) Rademacher law } \frac{1}{2}(\delta_{-1}+\delta_{+1});\\
   &
   \text{(2) Arcsine law } \frac{1}{\pi}\frac{1}{\sqrt{2-x^2}}\mathbf{1}_{\left[-\sqrt{2},\, \sqrt{2}\right]}\,\text{d}x;\\
   &\text{(3) Kesten-McKay law } \frac{1}{2\pi}\frac{\sqrt{4(1-1/2d)-x^2}}{1-x^2/2d}\mathbf{1}_{\left[-2\sqrt{1-1/2d},\,2\sqrt{1-1/2d}\right]}\,\text{d}x
\end{align*}
coincide with $\mu_w$ where, following the same order, $w$ is the graphon given by 
\begin{align*}
   &\text{(1) } w(x,y)=-1; \\
   & \text{(2) } w(x,y)=-\left(\mathbf{1}_{[0,1/2)\times [0,1/2)}(x,y)+\mathbf{1}_{[1/2,1]\times [1/2,1]}(x,y)\right);\\
   &\text{(3) } w(x,y)=-\sum_{i=1}^d\mathbf{1}_{[(i-1)/d,i/d)\times [(i-1)/d,i/d)}(x,y).
\end{align*}
\end{example}
\begin{proof}
The first case follows from the fact that the Rademacher law is the $q$-Gaussian law $\mu_q$ for $q=-1$. Note that the arcsine law can be written as $\frac{1}{\sqrt{2}}(\mu_{-1}\boxplus \mu_{-1})$, and more generally, the Kesten-McKay law can be written as $\frac{1}{\sqrt{d}}(\mu_{-1})^{\boxplus d}$; see \cite[Example 12.8, Exercise 12.21]{nica2006lectures}.
   We conclude by Proposition \ref{proposition: mu_w is closed under cla/free} and Remark \ref{remark: construction convolution}.
\end{proof}

\bibliographystyle{abbrvnat}
\bibliography{references.bib}

\begin{thebibliography}{34}
\providecommand{\natexlab}[1]{#1}
\providecommand{\url}[1]{\texttt{#1}}
\expandafter\ifx\csname urlstyle\endcsname\relax
  \providecommand{\doi}[1]{doi: #1}\else
  \providecommand{\doi}{doi: \begingroup \urlstyle{rm}\Url}\fi

\bibitem[Arizmendi et~al.(2025)Arizmendi, Mendoza, and
  Vazquez-Becerra]{ARIZMENDI2025110712}
O.~Arizmendi, S.~R. Mendoza, and J.~Vazquez-Becerra.
\newblock Bmt independence.
\newblock \emph{Journal of Functional Analysis}, 288\penalty0 (2):\penalty0
  110712, 2025.
\newblock ISSN 0022-1236.
\newblock \doi{https://doi.org/10.1016/j.jfa.2024.110712}.
\newblock URL
  \url{https://www.sciencedirect.com/science/article/pii/S0022123624004002}.

\bibitem[Borgs et~al.()Borgs, Chayes, Lovász, Sós, and
  Vesztergombi]{borgs2006counting}
C.~Borgs, J.~Chayes, L.~Lovász, V.~T. Sós, and K.~Vesztergombi.
\newblock \emph{Counting Graph Homomorphisms}, page 315–371.
\newblock Springer Berlin Heidelberg.
\newblock ISBN 9783540336983.
\newblock \doi{10.1007/3-540-33700-8_18}.
\newblock URL \url{http://dx.doi.org/10.1007/3-540-33700-8_18}.

\bibitem[Borgs et~al.(2008)Borgs, Chayes, Lov\'asz, S\'os, and
  Vesztergombi]{lovasz2008convergence1}
C.~Borgs, J.~T. Chayes, L.~Lov\'asz, V.~T. S\'os, and K.~Vesztergombi.
\newblock Convergent sequences of dense graphs. {I}. {S}ubgraph frequencies,
  metric properties and testing.
\newblock \emph{Adv. Math.}, 219\penalty0 (6):\penalty0 1801--1851, 2008.
\newblock ISSN 0001-8708,1090-2082.
\newblock \doi{10.1016/j.aim.2008.07.008}.
\newblock URL \url{https://doi.org/10.1016/j.aim.2008.07.008}.

\bibitem[Borgs et~al.(2012)Borgs, Chayes, Lov\'asz, S\'os, and
  Vesztergombi]{lovasz2012convergence2}
C.~Borgs, J.~T. Chayes, L.~Lov\'asz, V.~T. S\'os, and K.~Vesztergombi.
\newblock Convergent sequences of dense graphs {II}. {M}ultiway cuts and
  statistical physics.
\newblock \emph{Ann. of Math. (2)}, 176\penalty0 (1):\penalty0 151--219, 2012.
\newblock ISSN 0003-486X,1939-8980.
\newblock \doi{10.4007/annals.2012.176.1.2}.
\newblock URL \url{https://doi.org/10.4007/annals.2012.176.1.2}.

\bibitem[Bo{\.z}ejko(1998)]{bozejko1998completely}
M.~Bo{\.z}ejko.
\newblock Completely positive maps on {C}oxeter groups and the
  ultracontractivity of the q-{O}rnstein-{U}hlenbeck semigroup.
\newblock \emph{Banach Center Publications}, 43\penalty0 (1):\penalty0 87--93,
  1998.

\bibitem[Bo{\.z}ejko and Speicher(1991)]{bozejko1991example}
M.~Bo{\.z}ejko and R.~Speicher.
\newblock An example of a generalized {B}rownian motion.
\newblock \emph{Communications in Mathematical Physics}, 137:\penalty0
  519--531, 1991.

\bibitem[Bożejko and Speicher(1994)]{bozejko1994completely}
M.~Bożejko and R.~Speicher.
\newblock Completely positive maps on {C}oxeter groups, deformed commutation
  relations, and operator spaces.
\newblock \emph{Mathematische Annalen}, 300\penalty0 (1):\penalty0 97–120,
  Sept. 1994.
\newblock ISSN 1432-1807.
\newblock \doi{10.1007/bf01450478}.
\newblock URL \url{http://dx.doi.org/10.1007/BF01450478}.

\bibitem[Bożejko et~al.(1997)Bożejko, K\"{u}mmerer, and Speicher]{boejko1997}
M.~Bożejko, B.~K\"{u}mmerer, and R.~Speicher.
\newblock q -{G}aussian processes: Non-commutative and classical aspects.
\newblock \emph{Communications in Mathematical Physics}, 185\penalty0
  (1):\penalty0 129–154, Apr. 1997.
\newblock ISSN 1432-0916.
\newblock \doi{10.1007/s002200050084}.
\newblock URL \url{http://dx.doi.org/10.1007/s002200050084}.

\bibitem[Caspers and Fima(2017)]{caspers2017graph}
M.~Caspers and P.~Fima.
\newblock Graph products of operator algebras.
\newblock \emph{Journal of Noncommutative Geometry}, 11\penalty0 (1):\penalty0
  367--411, 2017.

\bibitem[Charlesworth and Collins(2021)]{charlesworth2021matrix}
I.~Charlesworth and B.~Collins.
\newblock Matrix models for $\varepsilon$-free independence.
\newblock \emph{Archiv der Mathematik}, 116\penalty0 (5):\penalty0 585--600,
  2021.

\bibitem[Charlesworth et~al.(2024)Charlesworth, de~Santiago, Hayes, Jekel, and
  Elayavalli]{charlesworth2024random}
I.~Charlesworth, R.~de~Santiago, B.~Hayes, D.~Jekel, and S.~K. Elayavalli.
\newblock Random permutation matrix models for graph products.
\newblock \emph{arXiv preprint arXiv:2404.07350}, 2024.

\bibitem[Chistyakov and G\"{o}tze(2008)]{chistyakov2008}
G.~P. Chistyakov and F.~G\"{o}tze.
\newblock Limit theorems in free probability theory. i.
\newblock \emph{The Annals of Probability}, 36\penalty0 (1), Jan. 2008.
\newblock ISSN 0091-1798.
\newblock \doi{10.1214/009117907000000051}.
\newblock URL \url{http://dx.doi.org/10.1214/009117907000000051}.

\bibitem[Ebrahimi-Fard et~al.(2017)Ebrahimi-Fard, Patras, and
  Speicher]{ebrahimifard2017}
K.~Ebrahimi-Fard, F.~Patras, and R.~Speicher.
\newblock ${\epsilon}$ -noncrossing partitions and cumulants in free
  probability.
\newblock \emph{International Mathematics Research Notices}, 2018\penalty0
  (23):\penalty0 7156–7170, May 2017.
\newblock ISSN 1687-0247.
\newblock \doi{10.1093/imrn/rnx098}.
\newblock URL \url{http://dx.doi.org/10.1093/imrn/rnx098}.

\bibitem[Green(1990)]{green1990graph}
E.~R. Green.
\newblock \emph{Graph products of groups}.
\newblock PhD thesis, University of Leeds, 1990.

\bibitem[Janson(1997)]{janson1997gaussian}
S.~Janson.
\newblock \emph{Gaussian {H}ilbert spaces}.
\newblock Number 129. Cambridge university press, 1997.

\bibitem[Jekel et~al.(2024)Jekel, Oussi, and Wysocza{\'n}ski]{jekel2024general}
D.~Jekel, L.~Oussi, and J.~Wysocza{\'n}ski.
\newblock General limit theorems for mixtures of free, monotone, and boolean
  independence.
\newblock \emph{arXiv preprint arXiv:2407.02276}, 2024.

\bibitem[Kunszenti-Kov\'acs et~al.(2022)Kunszenti-Kov\'acs, Lov\'asz, and
  Szegedy]{lovasz2022multigraph}
D.~Kunszenti-Kov\'acs, L.~Lov\'asz, and B.~Szegedy.
\newblock Multigraph limits, unbounded kernels, and {B}anach space decorated
  graphs.
\newblock \emph{J. Funct. Anal.}, 282\penalty0 (2):\penalty0 Paper No. 109284,
  44, 2022.
\newblock ISSN 0022-1236,1096-0783.
\newblock \doi{10.1016/j.jfa.2021.109284}.
\newblock URL \url{https://doi.org/10.1016/j.jfa.2021.109284}.

\bibitem[Lancien et~al.(2024)Lancien, Santos, and
  Youssef]{lancien2024centrallimittheoremtensor}
C.~Lancien, P.~O. Santos, and P.~Youssef.
\newblock Central limit theorem for tensor products of free variables, 2024.
\newblock URL \url{https://arxiv.org/abs/2404.19662}.

\bibitem[Lehner(2002)]{lehner2002free}
F.~Lehner.
\newblock Free cumulants and enumeration of connected partitions.
\newblock \emph{European Journal of Combinatorics}, 23\penalty0 (8):\penalty0
  1025--1031, 2002.

\bibitem[Lov\'asz(2012)]{lovasz2012largenetworks}
L.~Lov\'asz.
\newblock \emph{Large networks and graph limits}, volume~60 of \emph{American
  Mathematical Society Colloquium Publications}.
\newblock American Mathematical Society, Providence, RI, 2012.
\newblock ISBN 978-0-8218-9085-1.
\newblock \doi{10.1090/coll/060}.
\newblock URL \url{https://doi.org/10.1090/coll/060}.

\bibitem[Lov{\'a}sz and Szegedy(2006)]{lovasz2006limits}
L.~Lov{\'a}sz and B.~Szegedy.
\newblock Limits of dense graph sequences.
\newblock \emph{Journal of Combinatorial Theory, Series B}, 96\penalty0
  (6):\penalty0 933--957, 2006.

\bibitem[Magee and Thomas(2023)]{magee2308strongly}
M.~Magee and J.~Thomas.
\newblock Strongly convergent unitary representations of right-angled {A}rtin
  groups, 2023.
\newblock \emph{Preprint arxiv}, 2308, 2023.

\bibitem[Mingo and Speicher(2017)]{mingo2017free}
J.~A. Mingo and R.~Speicher.
\newblock \emph{Free probability and random matrices}, volume~35.
\newblock Springer, 2017.

\bibitem[Miyagawa and Speicher(2023)]{miyagawa2023dual}
A.~Miyagawa and R.~Speicher.
\newblock A dual and conjugate system for q-{G}aussians for all q.
\newblock \emph{Advances in Mathematics}, 413:\penalty0 108834, 2023.

\bibitem[M\l{}otkowski.(2004)]{mlotkowski2004lambdafree}
W.~M\l{}otkowski.
\newblock Λ-free probability.
\newblock \emph{Infinite Dimensional Analysis, Quantum Probability and Related
  Topics}, 07\penalty0 (01):\penalty0 27–41, Mar. 2004.
\newblock ISSN 1793-6306.
\newblock \doi{10.1142/s0219025704001517}.
\newblock URL \url{http://dx.doi.org/10.1142/S0219025704001517}.

\bibitem[Morampudi and Laumann(2019)]{morampudi2019many}
S.~C. Morampudi and C.~R. Laumann.
\newblock Many-body systems with random spatially local interactions.
\newblock \emph{Physical Review B}, 100\penalty0 (24):\penalty0 245152, 2019.

\bibitem[Nica and Speicher(2006)]{nica2006lectures}
A.~Nica and R.~Speicher.
\newblock \emph{Lectures on the combinatorics of free probability}, volume~13.
\newblock Cambridge University Press, 2006.

\bibitem[Shlyakhtenko(2006)]{shlyakhtenko2006remarks}
D.~Shlyakhtenko.
\newblock Remarks on free entropy dimension.
\newblock In \emph{Operator Algebras: The Abel Symposium 2004}, pages 249--257.
  Springer, 2006.

\bibitem[Speicher(1993)]{speicher1993generalized}
R.~Speicher.
\newblock Generalized statistics of macroscopic fields.
\newblock \emph{letters in mathematical physics}, 27:\penalty0 97--104, 1993.

\bibitem[Speicher and Wysocza{\'n}ski(2016)]{speicherjanusz2016mixture}
R.~Speicher and J.~Wysocza{\'n}ski.
\newblock Mixtures of classical and free independence.
\newblock \emph{Arch. Math.}, 107\penalty0 (4):\penalty0 445--453, Oct. 2016.

\bibitem[van Leeuwen and Maassen(1996)]{van1996obstruction}
H.~van Leeuwen and H.~Maassen.
\newblock An obstruction for q-deformation of the convolution product.
\newblock \emph{Journal of Physics A: Mathematical and General}, 29\penalty0
  (15):\penalty0 4741, 1996.

\bibitem[Voiculescu(1985)]{voiculescu1985symmetries}
D.~Voiculescu.
\newblock \emph{Symmetries of some reduced free product C*-algebras}, page
  556–588.
\newblock Springer Berlin Heidelberg, 1985.
\newblock ISBN 9783540395140.
\newblock \doi{10.1007/bfb0074909}.
\newblock URL \url{http://dx.doi.org/10.1007/BFb0074909}.

\bibitem[Voiculescu(1998)]{voiculescu1998analogues}
D.~Voiculescu.
\newblock The analogues of entropy and of {F}isher's information measure in
  free probability theory v. noncommutative {H}ilbert transforms: V.
  {N}oncommutative {H}ilbert transforms.
\newblock \emph{Inventiones mathematicae}, 132\penalty0 (1):\penalty0 189--227,
  1998.

\bibitem[Zhu(2019)]{zhu2019}
Y.~Zhu.
\newblock A graphon approach to limiting spectral distributions of
  {W}igner‐type matrices.
\newblock \emph{Random Structures \& Algorithms}, 56\penalty0 (1):\penalty0
  251–279, Oct. 2019.
\newblock ISSN 1098-2418.
\newblock \doi{10.1002/rsa.20894}.
\newblock URL \url{http://dx.doi.org/10.1002/rsa.20894}.

\end{thebibliography}

\end{document}